\documentclass[aap, preprint]{imsart}
\RequirePackage[OT1]{fontenc}
\RequirePackage{amsthm,amsmath}
\RequirePackage[numbers]{natbib}
\RequirePackage[colorlinks,citecolor=blue,urlcolor=blue]{hyperref}

\pubyear{2018}
\volume{28}
\issue{4}
\firstpage{2335}
\lastpage{2369}
\doi{10.1214/17-AAP1359}

\usepackage[english]{babel} 
\usepackage[utf8]{inputenc}
\usepackage[OT1]{fontenc}

\usepackage{amsmath,amsfonts,amssymb,amsthm}
\usepackage{mathrsfs}
\usepackage{xfrac}
\usepackage{enumitem}         
\setlist[itemize] {label=$\pmb\triangleright$}
\usepackage{stmaryrd} 

\usepackage{color}
\definecolor{darkgreen}{rgb}{0,0.55,0}
\usepackage[colorlinks,citecolor=blue,urlcolor=blue]{hyperref}

\startlocaldefs

\AtBeginDocument{%
   \def\MR#1{}
}

\newtheorem{thm}{Theorem}[section]

\newtheorem{cor}[thm]{Corollary}
\newtheorem{lem}[thm]{Lemma}

\theoremstyle{definition}
\newtheorem{defn}[thm]{Definition}

\newtheorem{ass}[thm]{Assumption}
\theoremstyle{remark}

\newtheorem{rmk}[thm]{Remark}
\newtheorem{example}[thm]{Example}

\numberwithin{equation}{section}

\renewcommand\*{\cdot}
\newcommand\1{\mathbf{1}}

\newcommand\R{\mathbb{R}}

\newcommand\N{\mathbb{N}}

\providecommand{\setc}[2]{\left\{ #1 \cond #2\right\}}
\providecommand{\abs}[1]{\left\lvert#1\right\rvert}
\providecommand{\norm}[2]{\left\lVert#1\right\rVert_{#2}}
\newcommand{\scal}[3]{\left\langle #1, #2\right\rangle_{#3}}

\newcommand{\PP}{\mathbb{P}}

\newcommand{\EE}{\mathbb{E}}

\newcommand{\cU}{\mathcal{U}}
\newcommand{\cA}{\mathcal{A}}
\newcommand{\cB}{\mathcal{B}}

\newcommand{\cC}{\mathcal{C}}
\newcommand{\cW}{\mathcal{W}}

\newcommand{\sD}{\mathscr{D}}

\newcommand\ddp{\tfrac{\partial}{\partial p}}

\newcommand\ddx{\tfrac{\partial}{\partial x}}
\newcommand\ddxx{\tfrac{\partial^2}{\partial x^2}}
\newcommand\ddy{\tfrac{\partial}{\partial y}}

\newcommand{\Ddx}{\partial_x}
\newcommand{\Ddxx}{\partial_{xx}}

\newcommand\F{\mathcal{F}}

\renewcommand\d{\,\operatorname{d}\hspace{-0.05cm}}

\providecommand\cond{\,\middle\vert\,}
\renewcommand\cond{\,\middle\vert\,}

\renewcommand{\P}[1]{\mathbb{P}\left[#1\right]}                     
\newcommand{\E}[1]{\mathbb{E}\left[#1\right]}                     

\newcommand\Id{\operatorname{Id}}                               
\newcommand\dom{\mathcal{D}}                                      

\newcommand\HS{\operatorname{HS}}
\newcommand\DA{\dom(\cA)}

\renewcommand\H{\mathfrak H}
\renewcommand\L{\mathfrak L}

\newcommand\cI{\mathcal I}

\newcommand\sign{\operatorname{sign}}

\newcommand{\bmu}{{\overline \mu}}
\newcommand{\bsigma}{{\overline \sigma}}
\newcommand{\VI}{\operatorname{VI}}
\newcommand{\llbrak}{\llbracket}
\newcommand{\rrbrak}{\rrbracket}

\endlocaldefs

\begin{document}

\begin{frontmatter}
\title{A stochastic Stefan-type problem\\ under first-order boundary conditions}
\runtitle{A stochastic Stefan-type problem}

\begin{aug}
  \author{\fnms{Marvin S.} \snm{M\"uller}\thanksref{t1,t2}\ead[label=e1]{marvin.mueller@math.ethz.ch}}

  \thankstext{t1}{Supported by the Swiss National Science Foundation through grant SNF $205121\_163425$.}
  \thankstext{t2}{Most of this work was carried out within the scope of the author's
    dissertation~\cite{diss} at TU Dresden with funding from the German Research Foundation (DFG) under grant ZUK 64.}

  \runauthor{M. S. M\"uller}
  
  \affiliation{ETH Z\"urich}
  
  \address{ETH Z\"urich\\
    Departement of Mathematics\\
    R\"amistrasse 101\\
    8092 Z\"urich\\
    Switzerland\\
    \printead{e1}\\
  }
\end{aug}

\begin{abstract}
 \, Moving boundary problems allow to model systems with phase transition at an inner boundary. Motivated by problems in economics and finance, we set up a price-time continuous model for the limit order book  and consider a stochastic and non-linear extension of the classical Stefan-problem in one space dimension. Here, the paths of the moving interface might have unbounded variation, which introduces additional challenges in the analysis. Working on the distribution space, It\={o}-Wentzell formula for SPDEs allows to transform these moving boundary problems into partial differential equations on fixed domains. Rewriting the equations into the framework of stochastic evolution equations and stochastic maximal $L^p$-regularity we get existence, uniqueness and regularity of local solutions. Moreover, we observe that explosion might take place due to the boundary interaction even when the coefficients of the original problem have linear growths.
\end{abstract}



\end{frontmatter}

Multi-phase systems of partial differential equations have a long history in applications to various fields in natural science and quantitative finance. Recent developments in modeling of demand, supply and price formation in financial markets with high trading frequencies ask for a mathematically rigorous framework for moving boundary problems with stochastic forcing terms. Motivated by this application, we consider a class of semilinear two-phase systems in one space dimension with first order boundary conditions at the inner interface. 

While the deterministic problems have been extensively studied in the
second half of the past century - see
e.\,g. \cite{lunardiMovingBoundary} and references therein - the stochastic equations are much less understood. In the past decade, several authors have started to study stochastic extensions of the classical Stefan problem. In 1888, Josef Stefan introduced this problem as a model for heat diffusion in the polar sea~\cite{stefanEis}. However, to the best of the author's knowledge so far the stochastic perturbations have been limited to the systems behaviour inside the respective phases or, in the context of Cahn-Hilliard equations~\cite{bloemkerCH}, on the boundary. As a step towards more realistic models, we also extend the Stefan-type dynamics of the free interface by Brownian noise which introduces additional challenges in the analysis. 

Barbu and da Prato~\cite{BarbuDaPratoStefan} used the so called enthalpy function in the setting of the classical Stefan problem with additive noise to transform the free boundary problem into a stochastic evolution equation of porous media type. In a series of papers, Kim, C.\,Mueller, Sowers and Zheng \cite{sowersEtAl2, sowersEtAl, zhengPhD} studied a class of linear stochastic moving boundary problems in one space dimension. After a coordinate transformation, the resulting SPDEs have been solved directly using heat kernel estimates. Extending these results, Keller-Ressel and M.\,S.\,M\"uller~\cite{SFBP} used classical estimates from interpolation theory to established a notion of strong solutions for stochastic moving boundary problems. The framework for existence, regularity and further analysis of the solution is based on the theory of mild and strong solutions of stochastic evolution equations in the sense of~\cite{dPZinf}. The change of coordinates was made rigorous by imposing a stochastic chain rule. Unfortunately, this way is no longer accessible when the path of the moving interface has unbounded variation or when the solution itself has a discontinuity at the inner boundary. Instead, we will switch into the space of generalized functions and use It\={o}-Wentzell formula for SPDEs to perform the transformation. It turns out that the terms describing the evolution of the density close to the boundary are distribution-valued, which gives the need for an extension of the concepts of solutions. The setting for analysis of the centered problems will still be based on semigroup theory for stochastic evolution equations, but can be located at the borderline case for existence.

Recently, various systems based on both stochastic and deterministic parabolic partial differential equations have been applied in finance as dynamical models for demand, supply and price formation, see e.\,g. \cite{bhqFunctional, bouchaudSMBP, lasrylions, markowichteichmannwolfram, bouchaudReactionDiffusion} which is just a short list and far away from being complete. For modern financial markets with high trading frequencies, we introduce a class of continuous models for the limit order book density with infinitesimal tick size, where the evolution of buy and sell side is described by a semilinear second-order SPDE and the mid price process defines a free boundary separating buy and sell side. Based on empirical observations~\cite{contImbalance, liptonImbalance}, we assume that average price changes can are determined by the bid-ask imbalance. Extending the models presented in~\cite{zhengPhD, SFBP} we allow the price process to have unbounded variation.

The paper is structured in the following way. In the first section, we introduce the moving boundary problem and the centered SPDE, define the notions of solutions and present the results on existence and regularity of local solutions and characterize its explosion times. Equations of this type arise, for instance, in macroscopical descriptions of demand and supply evolution in nowadays financial markets. In Section~\ref{sec:LOB}, we set up a dynamic model for the density of the so called limit order book. We then switch back to the analysis and solve the centered equation using existence theory for stochastic evolution equations in the framework of stochastic maximal $L^p$-regularity in Section~\ref{sec:centered}. This theory was studied in detail on a general class of Banach spaces by Weis, Veraar and van Neerven~\cite{weisMaxRegEvEq, weisMaxReg}. The required results, adapted to the Hilbert space setting, are sketched in the appendix. In Section~\ref{sec:trafo}, we switch to Krylov's framework of solutions of SPDEs in the sense of distributions, see \cite{krylov6, krylovItoWentzell}, and translate the existence and regularity results for the equations on the moving frames. In Section~\ref{sec:motivation}, we present heuristically a toy example which illustrates our notion of solutions for stochastic moving boundary problems.\\

Without further mentioning, we work on $(\Omega,\F, (\F_t),\PP)$, a
filtered probability space with the usual conditions. For a stopping
time $\tau$ we denote the closed stochastic interval by $\llbrak
0,\tau \rrbrak := \{(t,\omega)\in [0,\infty)\times \Omega\,\vert\,
t\leq \tau(\omega)\}$. Respectively, we define $\llbrak
0,\tau\llbrak$, $\rrbrak 0,\tau\llbrak$ and $\rrbrak
0,\tau\rrbrak$. For stochastic processes $X$ and $Y$ we say $X(t) =
Y(t)$ on $\llbrak 0,\tau\llbrak$, if equality holds for almost all
$\omega \in \Omega$ and all $t\geq 0$ such that $(t,\omega)\in \llbrak
0,\tau\llbrak$. Given Hilbert spaces $E$ and $H$, we write
$E\hookrightarrow H$ when $E$ is continuously and densely embedded
into $H$. As usual, we denote by $L^q$ the Lebesgue space, $q\geq 1$,
and with $H^s$, $s>0$, the Sobolev spaces of order $s>0$. Moreover,
$\ell^2(E)$ is the space of $E$-valued square summable sequences and
$\HS(U;E)$, for separable Hilbert spaces $U$ and $E$, is the space of
Hilbert-Schmidt operators from $U$ into $E$. The scalar product on $E$
will be denoted by $\scal{.}{.}{E}$. When working on the distribution
space $\sD$, we denote the dualization by $\scal{.}{.}{}$. We will
work only with real separable Hilbert spaces and implicitely use their
complexification when necessary to apply results from the
literature. We typically denote positive constants by $K$ which might
change from line to line and might have subindices to indicate
dependencies.

\section{A Stochastic Moving Boundary Problem}
\label{sec:spde}
We consider the stochastic moving boundary problem in one space dimension
\begin{gather}
  \begin{split}
    \d v(t,x) &= \left[\eta_+ \ddxx v +  \mu_+\left(x-p_*(t),p_*(t), v, \ddx v \right) \right] \d t\\
    & \qquad \qquad\qquad\qquad+ \sigma_+\left(x-p_*(t), p_*(t),  v\right)\d \xi_t (x), \quad  x > p_*(t),\\
    \d v(t,x) &= \left[\eta_- \ddxx v +\mu_-\left(x-p_*(t),p_*(t),  v, \ddx v \right) \right] \d t\\
    &\qquad\qquad\qquad\qquad + \sigma_-\left(x-p_*(t), p_*(t),  v\right) \d \xi_t(x), \quad  x < p_*(t),\\
    \d p_*(t) &= \varrho\Big(v(t,p_*(t)+), v(t,p_*(t)-)\Big)\d t + \sigma_* \d B_t,
  \end{split}\label{eq:smbp2}
\end{gather}
with Robin boundary conditions
\begin{gather}
  \begin{split}
    \ddp v(t,p_*(t)+) &= \kappa_+ v(t,p_*(t)+),\\
    \ddp v(t,p_*(t)-) &= -\kappa_-  v(t,p_*(t)-),
  \end{split}\label{eq:bc}
\end{gather}
for $t> 0$, $\mu_\pm: \R^4\rightarrow \R$, $\sigma_\pm: \R^3\rightarrow \R$, $\eta_\pm>0$, $\sigma_*\geq 0$, and $\kappa_+$, $\kappa_-\in [0,\infty)$. 
On $(\Omega,\F,(\F_t),\PP)$, $B$ is a real Brownian motion and $\xi$ the spatially colored noise,
\begin{equation}
  \label{eq:colorednoise}
  \xi_t(x) :=  T_\zeta W_t (x),\qquad T_\zeta w(x) := \int_\R \zeta(x,y) w(y) \d y,\quad x\in \R,\,t\geq 0,
\end{equation}
where $W$ is a cylindrical Wiener process $W$ on $U:=L^2(\R)$ with covariance operator $\Id$, independent of $B$, and $\zeta:\R^2\to\R$ an integral kernel. It was shown in~\cite{SFBP}, that for $\sigma_* = 0$, and $\kappa_- = \kappa_+ = \infty$, one can shift the equation onto the fixed domain $\dot \R:= \R \setminus\{0\}$, and there exists a strong solution of the integral equation corresponding to~\eqref{eq:smbp2}. Following this procedure at least informally, we get for $u(t,x):= v(t,x+p_*(t))$, $x\neq 0$, $t>0$,
\begin{gather}
  \begin{split}
    \d u(t,x) &= \left[(\eta_++\tfrac12\sigma_*^2) \ddxx u +  \mu_+\left(x,p_*(t), u, \ddx u \right)\vphantom{ \varrho(u(t,0+),u(t,0-))\ddx u(t,x)}\right.  \\
    &\qquad\qquad\qquad \left.+\vphantom{(\eta_++\tfrac12\sigma_*^2) \ddxx u +  \mu_+\left(x,p_*(t), u, \ddx u \right)+}  \varrho(u(t,0+),u(t,0-))\ddx u(t,x)\right] \d t\\
    & \qquad + \sigma_+\left(x, p_*(t),  u)\right)\d \xi_t (x+p_*(t)) + \sigma_* \ddx u(t,x) \d B_t, \quad  x > 0,\\
    \d u(t,x) &= \left[(\eta_-+\tfrac12 \sigma_*^2) \ddxx u +\mu_-\left(x, p_*(t),  u, \ddx u \right) \vphantom{ + \varrho(u(t,0+),u(t,0-))\ddx u(t,x)} \right.\\
    &\qquad\qquad\qquad + \left.\vphantom{(\eta_-+\tfrac12 \sigma_*^2) \ddxx u +\mu_-\left(x, p_*(t),  u, \ddx u \right)+} \varrho(u(t,0+),u(t,0-))\ddx u(t,x)\right] \d t\\
    &\qquad + \sigma_-\left(x, p_*(t),  u\right) \d \xi_t(x+p_*(t))+ \sigma_* \ddx u(t,x) \d B_t, \quad  x < 0,\\
    \d p_*(t) &= \varrho\Big(u(t,0+), u(t,0-)\Big)\d t + \sigma_* \d B_t,
  \end{split}\label{eq:spde}
\end{gather}
with boundary conditions
\begin{equation}
  \label{eq:bcc}
  \ddx u(t,0+) = \kappa_+ u(t,0+),\qquad       \ddx u(t,0-) = -\kappa_- u(t,0-).
\end{equation}
Problem~\eqref{eq:spde} admits
several features worth mentioning.
\begin{itemize}
\item Even when $\mu_+ = \mu_- = 0$, and $\sigma_+$, $\sigma_-$ are linear in $u$, the centered problem has a non-linearity, which is non-local in space and involving the first order derivative.
\item Due to the additional second order term the transformation seems to increase regularity and the equation is parabolic even for $\eta_+= \eta_- = 0$ as long as $\sigma_*>0$. 
\item When $\sigma_*>0$, the first derivative appears in the noise
  term. Recall that the Brownian noise scales differently from time
  and the equation is the borderline case where we could hope to get
  existence. In particular, even linear equations with gradient in the
  noise term can run out of parabolicity, see
  Remark~\ref{rmk:maximaility} and \cite{veraarPara}. Moreover, it
  seems that existence for~\eqref{eq:spde} cannot be shown in the
  present framework, when presuming Dirichlet boundary conditions at
  $x=0$, and replacing $u$ by $\ddx u$ in the dynamics of $p_*$.
\end{itemize}
We emphasize that the strong transformation procedure used
in~\cite{SFBP} does not work in this case, even if $\sigma_* =
0$. However, the behaviour of the solutions should be quite similar
when we restrict to a region away from the free interface $p_*$. The
only problem appearing is due to the discontinuity of $v$ at $p_*$. We
work around this problem by switching to the space of generalized
functions where we can apply Krylov's version of Ito-Wentzell
formula~\cite{krylovItoWentzell}. On this way, we obtain a description
of the evolution of $v$ around $p_*$, which will be part of
Definition~\ref{df:fbp}. We now focus on the centered
problem~\eqref{eq:spde}. Recall that $(u,p_*,\tau)$ is called local \emph{strong} solution
of~\eqref{eq:spde}, if $(u,p_*)$ is an $L^2(\R)\times \R$-predictable
stochastic process and $\tau$ a predictable stopping time such
that \eqref{eq:spde} holds true on $\llbrak 0,\tau\ \llbrak$, in the
sense of an $L^2(\R) \oplus \R$-integral equation, and~\eqref{eq:bcc} holds true $\d t\otimes \d
\PP$-almost everywhere. In particular, all the (stochastc) integrals
are assumed to exist on $L^2(\R)$ and $\R$, respectively. A solution
is called maximal, if there exists no solution on a stricly larger
stochastic interval. See also Definition~\ref{df:strongSEE} and
Section~\ref{sec:centered} for a more detailed formulation.

\begin{ass}
  \label{a:rho}
  $\varrho: \R^2\to \R$ is locally Lipschitz continuous.
\end{ass}
\begin{ass}
  $\mu_+$, $\mu_- : \R^4 \to \R$ fulfill \ref{ai:mugrowths},
  \ref{ai:mulipu} and \ref{ai:mulipp}.
  \label{a:mu}
  \begin{enumerate}[label=(\roman*)]
  \item\label{ai:mugrowths} There exist $a \in L^2(\R)$, $b \in L^\infty_{loc}(\R^2; \R)$ such that for all $x, y, z\in \R$
    \[ \abs{\mu_+ (x,p,y,z)}+ \abs{\mu_- (x,p,y,z)} \leq  b(y,p)\left(a(x) + \abs{y} + \abs{z}\right).\]
  \item\label{ai:mulipu} For all $R>0$ exist $L_R>0$ and $a_R\in L^2(\R)$ such that for all $x$, $y$, $\tilde y$, $z$, $\tilde z$, $p\in \R$, with $\abs{y}$, $\abs{\tilde y}$, $\abs{p}\leq R$, it holds that  
    \begin{multline*}
      \;\qquad\abs{\mu_\pm(x,p,y,z) - \mu_\pm(x,p,\tilde y,\tilde z)} \leq \\\leq L_R  (a_R(x) + \abs{z}+\abs{\tilde z})\abs{y-\tilde y} + L_R\abs{z-\tilde z}.      
    \end{multline*}
  \item\label{ai:mulipp} For all $R>0$ exist  $a_R\in L^2(\R)$ and $b_R > 0$, such that for all $x$, $y$, $z$, $\tilde z$, $p$, $\tilde p\in \R$, with $\abs{y}$, $\abs{p}$, $\abs{\tilde p}\leq R$, it holds that
    \begin{equation*}
      \abs{\mu_\pm (x,p,y,z)- \mu_{\pm}(x,\tilde p,y,z)} \leq \left( a_R(x) + b_R \left(\abs{y}+\abs{z}\right)\right) \abs{p-\tilde p}.      
    \end{equation*}
  \end{enumerate}
\end{ass}
\begin{ass}
  \label{a:sigma}
  Let $\sigma_* \geq 0$ and for all $p\in \R$, $\sigma_+(.,p,.)$, $\sigma_-(.,p,.)\in C^{1}(\R^2;\R)$. Moreover,
  \begin{enumerate}[label=(\roman*)]
  \item\label{ai:sigmagrowths} There exist $a\in L^2(\R_+)$ and $b$, $\tilde b \in L^\infty_{loc}(\R^2; \R_+)$ such that
    \[  \abs{ \sigma(x,p,y)} + \abs{\ddx \sigma(x,p,y)} \leq   b(y,p) \left(a(x) +\abs{y}\right).\]
    and
    \[ \abs{\ddy\sigma(x,p,y)} \leq \tilde b(y,p).\]
  \item\label{ai:sigmalipu} For all $R\in \N$ exists $L_R>0$ such that for all $x$, $y$, $\tilde y$, $p\in \R$, with $\abs{y}$, $\abs{\tilde y}$, $\abs{p}\leq R$,
    \begin{align*}
      \abs{\sigma_\pm(x,p,y) - \sigma_\pm(x,p,\tilde y)} &\leq  L_R \abs{y-\tilde y},\\
      \abs{\ddx \sigma_\pm(x,p,y) - \ddx \sigma_\pm(x,p,\tilde y)} &\leq  L_R \abs{y-\tilde y},\\
      \abs{\ddy \sigma_\pm(x,p,y) - \ddy \sigma_\pm(x,p,\tilde y)} &\leq  L_R \abs{y-\tilde y}.
    \end{align*}
  \item\label{ai:sigmalipp} For all $R\in \N$ exists an $a_R\in L^2(\R)$ such that for all $x$, $y$, $\tilde y$, $p\in \R$, with $\abs{y}$, $\abs{\tilde y}$, $\abs{p}\leq R$, it holds for $i\in \{0,1\}$,
    \[\abs{\tfrac{\partial^{(i)}}{\partial x^{i}}\sigma_\pm(x,p,y) - \tfrac{\partial^{(i)}}{\partial x^{i}}\sigma_\pm(x, \tilde p,y)} \leq b_R(a_R(x) + \abs{y})\abs{p-\tilde p},\]
    and
    \[\abs{\ddy \sigma_\pm(x,p,y) - \ddy \sigma_\pm(x,\tilde p,y)} \leq \tilde b_R \abs{p-\tilde p}.\]
  \end{enumerate}
\end{ass}
\begin{ass}
  \label{a:zeta}
  $\zeta(.,y) \in C^2(\R)$ for all $y\in \R$ and $\tfrac{\partial^{i}}{\partial x^i}\zeta(x,.)\in L^2(\R)$ for all $x\in \R$, $i\in\{0,1,2\}$. Moreover, 
  \begin{equation}
    \sup_{x\in \R} \norm{\tfrac{\partial^{i}}{\partial x^i} \zeta(x,.)}{L^2(\R)} <\infty,\quad i\in\{0,1,2\}.\label{eq:Azeta}
  \end{equation}
\end{ass}

\begin{thm}\label{thm:euspde}
  Let Assumptions~\ref{a:rho}, \ref{a:mu}, \ref{a:sigma} and~\ref{a:zeta} hold true. Then, for all $\F_0$-measurable initial data $(u_0,p_0) \in H^1(\dot \R)\times \R$, there exists a unique maximal strong solution $(u,p_*,\tau)$ of~\eqref{eq:spde} with paths almost surely in $L^2([0,\tau); H^2(\dot \R))\cap C([0,\tau); H^1(\dot \R))$. 
\end{thm}
\begin{rmk}\label{rmk:Besov}
  When imposing stronger assumptions on the initial data we would expect also more spatial regularity of the solution. More detailed, when $u_0$ takes almost surely values in the Besov space \(B^{2 - \frac2{q}}_{2,q}(\dot\R)\) for some $q>2$, then, we would expect $u$ to have almost surely paths in
  \[L^q([0,\tau); H^2(\dot\R)) \cap C([0,\tau); B^{2-\frac2{q}}_{2,q}(\dot \R)).\]
  This indeed follows from results on stochastic maximal $L^q$-regularity, see~\cite[Thm 3.5]{weisMaxRegEvEq}, but only in the case when $\eta_+$ and $\eta_-$ are chosen sufficiently large or $\sigma_*$ is sufficiently small. 
  In this case the results of this section also hold true when the locally bounded function $b$ in Assumption~\ref{a:mu}.\ref{ai:mugrowths} is in $L_{loc}^\infty(\R^3)$ and depends on $p$, $y$ and $z$. To this end, one has to choose $q>4$ so that the Besov space above is embedded in $BUC^1(\dot \R)$. In this case, we need to assume that $u_0 \in B^{2- \frac2{q}}_{2,q} (\dot \R)$ and $u_0$ satisfies~\eqref{eq:bc}. Unfortunately, exact bounds on $\eta_+$ and $\eta_-$ or $\sigma_*$ have to be computed in terms of the constants $M_q$, $M_q^W$ in Appendix~\ref{appendix}. See Lemma~\ref{lem:noiselip} and Remark~\ref{rmk:para} for the issue concerning the impact of these constants.
\end{rmk}
Let us now consider some examples which might be of interest in applications. 
\begin{example}
  Let $\zeta(x,y) := \zeta(x+y)$ for $\zeta \in H^2(\R)\cap C^2(\R)$. Assume that $\kappa_+ = \kappa_->0$, $\varrho(x,y) = \rho\*(y-x)$ for some $\rho>0$ and $\sigma(x,p,y) := \sigma y$ for $\sigma \neq 0$. In this case, we can replace $v(t,p_*(t)\pm)$ in~\eqref{eq:smbp2} by the first derivatives and get an extension of the stochastic Stefan(-type) problems considered in \cite{SFBP} and \cite{sowersEtAl}, but with Robin instead of Dirichlet boundary conditions.
\end{example}
\begin{example}\label{ex:burger}
  In the setting of the previous example, let $\mu(x,p,y,z) := y\*z$. Then, the solution $v$ behaves like a stochastic viscous Burger's equation inside of the phases. 
\end{example}
\begin{example}
  With the modifications and limitations mentioned in Remark~\ref{rmk:Besov} one could also cover non-linearities of the form $\mu(x,p,y,z) = z^2$, or any other polynomial in $y$ and $z$.
\end{example}

To handle the randomly moving frames on which the solutions of~\eqref{eq:smbp2} are expected to live, we define the function spaces, for $x\in\R$,
\begin{equation}\label{eq:moving_frame}
  \Gamma^1(x) := \left\{v: \R \to \R\mid  \left.v\right|_{\R\setminus\{x\}} \in H^1(\R\setminus\{x\})\right\},
\end{equation}
and 
\begin{equation}\label{eq:moving_frame}
  \Gamma^2(x) := \left\{v: \R \to \R\mid  \left.v\right|_{\R\setminus\{x\}} \in H^2(\R\setminus\{x\}),\; v(x+.) \text{ satisfies~\eqref{eq:bcc}}\right\}.
\end{equation}
In the following, we denote the first two spatial weak derivatives on $\R\setminus\{x\}$ by $\nabla$ and $\Delta$, respectively. 

By Sobolev embeddings, any $v\in \Gamma^k(x)$ can be identified with an element of $BUC^{k-1}(\R\setminus\{x\})$ and, since $\{x\}$ has mass $0$, also of $L^2(\R)$. For all $v\in \Gamma^1(x)$ it also holds that $\nabla v$ and, if $v\in \Gamma^1(x)$, then $\Delta v$ are elements in $L^2(\R)$, again. 
To shorten the notation, we introduce the functions $\bmu: \R^5\rightarrow \R$, $\bsigma: \R^3\rightarrow \R$,
\begin{equation}
  \label{eq:bmu}
  \bmu(x,p,v,v',v''):=
  \begin{cases}
    \eta_+ v'' + \mu_+(x,p,v,v'),  & \quad x>0,\\
    \eta_- v'' + \mu_-(x,p,v,v'), & \quad  x<0,
  \end{cases}
\end{equation}
and
\begin{equation}
  \label{eq:bsigma}
  \quad
  \bsigma(x,p,v) := 
  \begin{cases}
    \sigma_+(x,p,v), & \quad  x>0,\\
    \sigma_-(x,p,v), & \quad  x<0.
  \end{cases}
\end{equation}
Denote by $\delta_x$ the Dirac distribution with mass at $x\in \R$ and by $\delta_x'$ its derivative. We define the following functions, which take values in the space of distributions,
\begin{alignat*}{2}
  L_1 &\colon \bigcup_{x\in \R} \left(\Gamma^1(x) \times \{x\}\right) \to \sD,&\quad (v,x) &\mapsto -(v(x+) - v(x-)) \delta_x, \\
  L_2 &\colon \bigcup_{x\in \R} \left(\Gamma^2(x) \times \{x\}\right) \to \sD,&\quad (v,x) &\mapsto (v(x+) - v(x-)) \delta'_x \\
  &&&\qquad - (\nabla v(x+) -\nabla v(x-)) \delta_x.  \nonumber
\end{alignat*}
\begin{rmk}\label{rmk:rest}
  Given $x\in \R$, $v\in \Gamma^2(x)$ it holds that
  \[L_1(v,x)\vert_{\R\setminus\{x\}} = L_2(v,x) \vert_{\R\setminus\{x\}} = 0.\]
  Here, for $f\in \sD$ and a Borel set $I\subset \R$ we write $f\vert_I = 0$ when 
  \[ \scal{f}{\phi}{} = 0,\qquad \forall\, \phi \in C^\infty(I).\]
\end{rmk}
Transforming~\eqref{eq:spde} back into the moving boundary problem we
observe that the description in~\eqref{eq:smbp2} does not explicitly
tell us how $t\mapsto v(t,x)$ should behave when $p_*(t) = x$. Note
that, if $\sigma_*>0$, then for each $T\in (0,\infty)$ the event
$p_*(t) = x$ occurs either for infinitely many $t\in
[0,T]$ or for none. Motivated by the discussion in Section~\ref{sec:motivation}, the following definition of notion of a solution for~\eqref{eq:smbp2} is the ``natural'' definition, in the sense that $L_1$ and $L_2$ describe the behaviour of $v$ at $p_*$. On the other side, Remark~\ref{rmk:rest} shows how to recover~\eqref{eq:smbp2}.

\begin{defn}\label{df:fbp}
  A \emph{local solution} of the stochastic moving boundary problem~\eqref{eq:smbp2} with initial data $v_0$ and $p_0$, is an $L^2(\R)\times\R$ predictable process $(v,p_*)$, and a positive and predictable stopping time $\tau$, with 
  \begin{equation*}
    (v,p_*):\; \llbrak 0, \tau \llbrak \to \bigcup_{x \in \R} \left( \Gamma^1(x)\times \{x\} \right) \; \subseteq \;  L^2(\R)\times \R,
  \end{equation*}
  such that $v(t,.)$ is $\Gamma^2(p_*(t))$-valued, and, for all $\phi \in C^\infty_0(\R)$ on $\llbrak 0, \tau \llbrak$,
  \begin{align*}
    \scal{v(t) - v_0}{\phi}{} &= \int_0^t \scal{\bmu(.-p_*(s),p_*(s), v(s), \nabla v(s), \Delta v(s))}{\phi}{} \d s\\
                              &\quad+ \int_0^t \scal{\bsigma(.-p_*(s), p_*(s), v(s))\d\xi_s}{\phi}{}\\
                              &\quad+ \int_0^t \scal{L_1(v(s), p_*(s))}{\phi}{} \d p_*(s)\\
                              &\quad+ \tfrac12 \int_0^t \scal{L_2(v(s),p_*(s))}{\phi}{} \d\, [p_*](s),\\
    p_*(t) - p_0  &=  \int_0^t \varrho\left(v(s, p_*(s)+), v(t,p_*(s)-)\right) \d s + \sigma_* B_t.
  \end{align*}
  The solution is called global, if $\tau = \infty$ a.\,s. and the interval $\llbrak 0, \tau \llbrak$ is called maximal if there is no solution of \eqref{eq:smbp2} on a larger stochastic interval.
\end{defn}
\begin{rmk}
  The stochastic integral term is defined as
  \begin{equation}
    \label{eq:13}
    \int_0^t \scal{\bsigma(., p_*(s), v(s))\d\xi_s}{\phi}{} = \sum_{k=1}^\infty  \int_0^t\scal{\bsigma(.,p_*(s),v(s)) T_{\zeta}e_k}{\phi}{}  \d \beta^k_s,
  \end{equation}
  which is implicitly assumed to exist in $L^2(\Omega)$.
\end{rmk}
\begin{rmk}
  The quadratic variation of $p_*$ is $[p_*](t) = \sigma_*^2 t$, $t\geq 0$.
\end{rmk}
\begin{rmk}
  This notion of solution is not exactly what one would typically expect under a \emph{weak} or \emph{distributional} solution. In contrast to~\cite[Definition 3.2]{sowersEtAl}, we require $\nabla v$ and $\Delta v$ to exist as $L^2(\R)$ elements, which assures analytically strong existence for the centered equations. 
\end{rmk}

\begin{thm}\label{thm:EUsmbp}
  Let Assumptions~\ref{a:rho}, \ref{a:mu}, \ref{a:sigma} and~\ref{a:zeta} hold true. Let $p_0\in \R$ and $v_0\in H^1(\R\setminus\{p_0\})$ be $\F_0$-measurable and $(u,p_*,\tau)$ be the unique maximal strong solution of~\eqref{eq:spde} with initial data $(v_0(.+p_0), p_0)$ and set $v(t,x) := u(t,x-p_*(t))$, $t\geq 0$, $x\in \R$. Then, $(v,p_*, \tau)$ is a local solution of~\eqref{eq:smbp2} in the sense of~\ref{df:fbp} and satisfies $v$, $\nabla v\in C([0,\tau);L^2(\R))$, $p_*\in C([0,\tau); \R)$ and $\Delta v\in L^2([0,\tau); L^2(\R))$ almost surely. Moreover, $(v, p_*,\tau)$ is unique and maximal under all such solutions. 
\end{thm}

\begin{rmk}
  In general, we cannot expect $x\mapsto v(t,x)$ to be continuous at
  $x=p_*(t)$, even if this holds true for $t=0$. For instance,
  consider the special situation where $\kappa_+ = \kappa_- = 0$ and
  $\mu_+$, $\mu_- \equiv 0$, $\rho \equiv 0$, $\sigma_+$, $\sigma_-\equiv 0$, $\sigma_* > 0$ and let $v_0 \in C(\R)\cap H^1(\R)$. 

  Assume that $v(t,.) \in C(\R)$, $\d t\otimes
  \PP$ almost everywhere. Then by Neumann boundary conditions at
  $p_*$ we have also $v(t,.)\in C^1(\R)\cap H^2(\R)$ and 
  \[L_1(v(t.), p_*(t)) = L_2(v(t,.), p_*(t)) = 0.\]
  Thus, Definition~\ref{df:fbp} implies that $v$ is the weak solution
  of the (deterministic) heat equation on $\R$. 
  Since $p_*(t) = \sigma_* B_t$ and $v$ is independent of $p_*(t)$, we get
  from~\eqref{eq:bc} that $\ddx v(t,x) = 0$ for almost all $t>0$ and
  $x\in \R$. Because $v(t,.)\in L^2(\R)$ this also yields $v(t,x) = 0$.  
\end{rmk}

\begin{rmk}
  In the deterministic situation, i.\,e. $\sigma_+$,
  $\sigma_-\equiv 0$ and $\sigma_* = 0$, one can get local classical
  solutions of the centered problem~\eqref{eq:spde} e.\,g. by standard theory for semilinear evolution
  equations~\cite{lunardiAnalytic}. Using time-differentiability
  of $u$ and $p_*$, the change of coordinates $x\mapsto x+p_*$ can then be performed by chain rule, locally on $[0,t_x)$, where
  \[ t_x:=  \inf\{t>0 \,\colon\, p_*(t) = x \}\wedge \tau,\qquad x\in \R.\]  
\end{rmk}

\subsection{Global Solutions}
Assume that the assumptions of Theorem~\ref{thm:euspde} are satisfied and let $(u,p_*,\tau)$ be the unique maximal solution of~\eqref{eq:spde}. Define the stopping time
\[ \tau_0 := \limsup_{N\to\infty}\,\inf\left\{ t\geq 0 \,\mid \, t<\tau,\, \abs{u(t,0+)} + \abs{u(t,0-)} > N\right\}. \] 
Here, we here use the convention that $\inf \emptyset := \infty$. The following assumption ensures that the coefficients involving $\mu$ and $\sigma$ have linear growth and hence, one would expect that explosion in~\eqref{eq:spde} can be due to the moving inner boundary only. 
\begin{ass}
  \label{a:musigmalg}
  Assume that $b$ and $\tilde b$ in Assumption~\ref{a:mu} and~\ref{a:sigma} are globally bounded.
\end{ass}
\begin{thm}\label{thm:tau0tau}
  Let the assumptions of Theorem~\ref{thm:euspde} be satisfied, and Assumption~\ref{a:musigmalg} hold true. Then, $\P{\tau_0 = \tau} = 1$, and almost surely on $\{\tau<\infty\}$,
  \[ \lim_{t\nearrow \tau} \abs{v(t,p_*(t)+)} + \abs{v(t,p_*(t)-)} = \infty.\]
\end{thm}
\begin{thm}\label{thm:global}
  Assume that the assumptions of Theorem~\ref{thm:euspde} are satisfied, that Assumption~\ref{a:musigmalg} holds true, and that $\varrho$ is globally bounded. Then, $\tau = \infty$ almost surely.  If, moreover, $v_0$, $\nabla v_0 \in L^2(\Omega\times \R, \d\PP\otimes \d x)$ and $p_0 \in L^2(\Omega; \R)$, then for all $T>0$ there exists a constant $K$ such that
  \begin{multline}
    \EE\sup_{0\leq t\leq T} \left(\abs{p_*(t)}^2 + \int_\R \abs{v(t,x)}^2 + \abs{\nabla v(t,x)}^2\d x\right)\\
    + \EE\int_0^T \int_\R \abs{\Delta v(t,x)}^2 \d x \d t  \\ 
    \leq K\left(1 + \E{\norm{v_0}{L^2}^2 + \norm{\nabla v_0}{L^2}^2 + \abs{p_0}^2}\right).
  \end{multline}
\end{thm}

\section{Application: Limit Order Book Models}
\label{sec:LOB}
In electronic trading, buy and sell orders of market participants are matched and cleared, or, if there is no counterpart, accumulated in the order book. Agents can either send market orders, which are executed ``immediately'' against the best orders currently available, or trade by limit orders. Limit buy orders are executed only at a specified price level $p$ at most and similarly, limit sell orders are executed only for a price $p$ or less. This price level $p$ is called \emph{limit}, and the minimal distance between two limits which is allowed in the market is called \emph{tick size}. One can think of the order book as a collection of buckets, indexed by the limits $p$ and each containing the number of limit orders active at time $t$. Actually, the orders in the book might get executed against incoming market or limit orders but also cancelled, which might happen in a substantial amount, especially in markets with high frequency trading~\cite{contHF}.

Price formation is now an extraction from the current order book state. The highest limit for which the order book contains buy orders is called bid price and the smallest limit with sell orders is called ask price. In highly liquid markets the difference of both, called \emph{spread}, is typically rather small and we assume it to be $0$ for our model. We stress the fact that the bid and ask prices are actually separating buy and sell side of the order book.

To abstract from this point we understand the order book as a two-phase system and aim to model the macroscopic behaviour of a highly liquid market under presence of high frequency trading. For a more detailed introduction and an overview on various types of models we refer to the survey~\cite{lobsurvey}, but see also \cite{bouchaudSMBP} for an approach more related to the present framework.

We denote by $v(t,x)$ the density of the limit buy and sell orders
placed at price $x$, which is on logarithmic scale. We keep the notion
``price'' when actually meaning logarithmic price. As a convention, buy orders will have a negative and sell orders a
positive sign. We let the tick size and the time discretization go to
$0$ and consider a price-time continuous approximation. Then, we
expect the evolution of the order book density to be described by an
SPDE, whereas the price process $p_*$ is the inner boundary
separating buy and sell side of the order book.
For the density dynamics, Zheng~\cite{zhengPhD} proposed the linear heat equation with additive space-time white noise $\xi$ and Dirichlet boundary conditions,
\begin{align*}
  \d v(t,x) &= \eta_+ \ddxx v(t,x) \d t + \sigma_+(\abs{x-p_*(t)}) \d \xi_t(x),\qquad x>p_*(t),\\
  \d v(t,x) &= \eta_- \ddxx v(t,x) \d t + \sigma_-(\abs{x-p_*(t)}) \d \xi_t(x),\qquad x<p_*(t),\\
  v(t,p_*(t)) &= 0,
\end{align*}
and the interaction of price and order book evolution is given by the Stefan condition
\[ \d p_*(t) = \rho\*(\ddx v(t,p_*(t)-) - \ddx v(t,p_*(t)+)) \d t.\]

In order to ensure existence of the right hand side under presence of
space-time white noise, it is assumed by Zheng that the volatilities
$\sigma_+$ and $\sigma_-$ vanish at $p_*$ faster than linear. To be
more precise, $\sigma_+$, $\sigma_-$ are Lipschitz continuous and
$\sigma_{\pm} \sim x^{\alpha}$ as $x\to 0$, for $\alpha >\sfrac32$
which in fact yields that $\sigma_{\pm}$ and its derivative vanish at
the origin. However, empirical observations show that the order flow
has a maximum at the bid and ask, see e.\,g. \cite{cont2010obmodel,
  lehalle}. The assumption on $\sigma$ would force us to average out
all events at the best bid and ask, in particular effects coming from
market orders. 

When introducing spatial correlation in the driving Gaussian field $\xi$, the assumptions on the decrease of the volatility can be relaxed, see Assumption~\ref{a:sigma} for instance. 

\begin{rmk}\label{rmk:lobslope}
  Using empirical data from Paris Bourse (now Euronext), Bouchaud et al.~\cite{bouchaudProfile} identified an average order book shape, which turns out to be symmetric at the price and has its maximum few ticks away from the bid (resp. ask) price. Note that it is not surprising that the maximum is not achieved directly the bid and ask levels since orders at the bid and ask level have a much higher probability being executed. As the distance $\delta$ from the bid (resp. ask) gets large, the average shape decreases like $\delta^{-\beta}$ for $\beta \approx 1.6$~\cite{bouchaudProfile}.
\end{rmk}

\subsection*{Order book dynamics} 
We split market participants in two major groups. On one side, we consider market makers trading at high frequencies and, on the other hand, low frequency traders such as institutional investors. The reason for this choice is that their trading behaviour and also their objectives are substantially different. For instance, market makers and high frequency traders typically do not accumulate large inventories~\cite{kirilenkoHFT}. In fact, what is typically observed at the end of each trading day is a rapid increase of the trading volume~\cite[Figure 2]{biaisLOBOrderFlow} since many positions have to be cleared overnight. On the other hand, the objectives of low frequency traders are based more on long-term strategies. For instance, an institutional investor has to sell a certain amount of stocks due to exogenous events. 

We consider the following dynamics for the evolution of the order book density $v$ and the price process $p_*$,
\begin{gather}
  \begin{split}
    \d v(t,x) &= \left[\eta_+ \ddxx v +  f_+(\abs{x-p_*(t)},v) \ddx v(t,x) - \alpha_+ v(t,x)\right] \d t \\
    & \qquad  + g_+(\abs{x-p_*(t)})\d t + \sigma_+(\abs{x-p_*(t)}) v(t,x)\d \xi_t (x), \;  x > p_*(t),\\
    \d v(t,x) &= \left[\eta_- \ddxx v - f_-(\abs{x-p_*(t)}, v)\ddx v(t,x)  - \alpha_- v(t,x) \right] \d t\\
    &\qquad  - g_-(\abs{x-p_*(t)}) \d t  + \sigma_-(\abs{x-p_*(t)}) v(t,x)\d \xi_t(x), \;  x < p_*(t),
  \end{split}\label{eq:smbplob}
\end{gather}
where the noise $\xi_t(x)$ is white in time but colored in space. Let
us motivate the terms of~\eqref{eq:smbplob} separately. Note that for
the effects which are due to market makers we mainly follow
\cite{bouchaudSMBP} but also include stochastic forcing terms which
have contribution to the total volume.
\begin{itemize}
\item A large amount of transactions due to market makers and high frequency traders is covered by cancellations or readjustments of orders in the book. On one hand, the individual adjustments average out and yield diffusive behaviour in the order book, described by the diffusion coefficients $\eta_+$ and $\eta_-$, see also \cite{bouchaudSMBP}.
\item On the other hand, following \cite{bouchaudSMBP}, collective readjustments of market participants are due to public available information. Here, this happens at rate $f_{+\slash -}(\abs{x-p_*(t)}, v(t,x))$. Assuming that orders are tendentially shifted into the direction of the mid price, we expect $f_+$, $f_- \geq 0$.
\item $\alpha_+$ and $\alpha_-$ are the cancellation rates for buy and sell side respectively. We assume that cancellation in the order book is proportional to the number of orders at the respective level.
\item The average limit order arrivels of institutional or private investors are modelled by $g_+$ and $g_-$. We assume their contribution is due to external forces and independent of the order book state, see also~\cite[Fig 1]{lehalle}.
\item $\sigma_+$ and $\sigma_-$ are the volatilities of proportional trading activity which is not averaged out in the model. In particular, $\sigma_{\pm}$ is allowed to depend on the distance to the mid price and thus, could incorporate higher order submission and cancellation rates close to the mid price. On the other hand, the impact of the noise will vanish far away from the mid price. Moreover, the empirical data in \cite{lehalle} indicate that volatility in the queues increases with lengths of the queues. 
\end{itemize}
For a detailed explanation of the diffusive drift behaviour see~\cite[Appendix 1]{bouchaudSMBP}.
\begin{example} [Burger's equation]
  We impose the assumption that market makers, or high frequency traders, in general, are the more tempted to move their order, the worse their actual position in the current queue is. Simplifying this, the rate at which orders are rearranged collectively should be proportional to the amount of orders in the respective bucket. Mathematically, this corresponds to the choice $f_{+\slash -}(x,v) := c_{+\slash -} v$, for $x$, $v\in \R$, $c_+>0$, $c_- <0$, and we get an extension of the classical viscous Burger's equation, see also Example~\ref{ex:burger}.
\end{example}

\begin{example}
  Extending the model we replace the assumption that collective readjustments tend into the direction to the mid price by the following. Agents with a position at the end of the queue aim to get into a better position and thus readjust their order to a price level with a shorter queue. To capute this we make the sign in front of the rate $f_{+\slash -}$ dependent on $\ddx v$, namely
  \[ f_+(x, v, v') = \sign(v')\tilde f_+(x,v) ,\qquad f_-(x, v, v') = \sign(v')\tilde f_-(x,v),\]
  for readjustment rates $\tilde f_{+\slash -}$. Equivilently, we replace $\ddx v(t,x)$ by $\abs{\ddx v(t,x)}$ in~\eqref{eq:smbplob}. In particular, under sufficient assumptions on $f_+$ and $f_-$, the extended model still fits into our analytic framework.
\end{example}

\subsection*{Price dynamics}

A commonly used predictor for the next price move is the imbalance of the order volume in the top level bid and ask queue, which we denote by $\VI$. Despite empirical evidence, see \cite{liptonImbalance} but also \cite{contImbalance} and the references therein, this mechanism is quite intuitive from a microscopic view point: For instance, if $\VI\gg 0$ which means that the volume at the best ask level is small compared to the best bid queue, then it should be much more likely that these orders are executed before the limit orders in the bid queue are. In this case the price moves up. With the same arguments, we would expect the price to decrease if $\VI\ll 0$. Translated to macroscopic scale this means
\begin{equation}
  \label{eq:priceVI}
  \d p_*(t) \approx \varrho(\VI(t)) \d t,
\end{equation}
at least on average, cf. \cite[Fig 1]{liptonImbalance}. Here, $\varrho: \R\to\R$ is a locally Lipschitz function with $\varrho(0) = 0$, describing the intensity of this relation. Recalling the convention that buy orders have negativ signs the volume imbalance reads as
\[\VI(t) = -v(t,p_*(t)-)-v(t,p_*(t)+).\]
Incorporating also exogenous events affecting price movements we perturb the price dynamics by Brownian noise,
\begin{equation}
  \label{eq:7}
  \d p_*(t) = \varrho(-v(t,p_*(t)-) -v(t,p_*(t)+))\d t + \sigma_* \d B_t,\quad t\geq 0.
\end{equation}
For $\sigma_*>0$, this can be seen as a
(time-inhomogeneous) extension of the classical Bachelier model. If
$\sigma_*=0$, this is a modification for first order boundary
conditions of the Stefan-type dynamics proposed in~\cite{zhengPhD} and~\cite{SFBP}. Note
that the additional noise term here is a significant step in direction
of more realistic models. In fact, the paths of the price processes
resulting for $\sigma_*=0$, also in the literature, are almost surely
$C^1$. Unfortunately, the analysis of the solutions - starting with
existence and uniqueness - gets much more involved. 

\begin{rmk}
  Prices here are on logarithmic scale. For small tick
  sizes, this is a reasonable approximation of the linear tick scale in
  real markets, cf. \cite{doubleauc}. On the other hand, it is also consistent with asymmetries
  between larger up and down moves of the price. In particular, the
  model assumptions can be chosen symmetric with respect to the mid price, here.
\end{rmk}
\begin{rmk}
  Recall that the Stefan condition occurred in a model for heat diffusion in a system of water and ice~\cite{stefanEis}. Since a certain amount of energy is required for solidification of water or melting of ice, conservation of mass only holds for the energy, and the temperature is diffusing only partially between ice and water. This is quite related to our situation, where only part of the agents having orders in the best bid bucket are willing to cross the spread and trade with market instead of limit orders. In particular, ice melts or water solidifies only when the enthalpy crosses a certain energy level, whereas the price moves only when the best bid, resp. ask bucket gets empty. 
\end{rmk}

\section{The centered equations}
\label{sec:centered}
By reflection, we can rewrite~\eqref{eq:spde} into the stochastic evolution equation on $\L^2:= L^2(\R_+)\oplus L^2(\R_+)\oplus \R$,
\begin{equation}
  \d X(t) = \left[ \cA X(t) + \cB(X(t))\right] \d t + \cC(X(t)) \d \cW_t,\qquad X(0) = X_0. \label{eq:SEEqG}
\end{equation}
with coefficients
\begin{align}
  \cA &= \begin{pmatrix}
    (\eta_++\tfrac12 \sigma^2_*) \Delta_+ & 0 & 0 \\ 
    0 & (\eta_-+\tfrac12 \sigma^2_*) \Delta_-  & 0 \\
    0 & 0 &0
  \end{pmatrix} - c\, \textrm{Id}, \label{eq:coefAG}\\
  \cB(u)(x) &= \begin{pmatrix}
    \mu_+(x, p,u_1(x), \ddx u_1(x))   \\
    \mu_-(-x, p,u_2(x), -\ddx u_2(x))  \\
    0
  \end{pmatrix} + c\, \textrm{Id} + \varrho(\cI(u)) \overline{\nabla} u ,\label{eq:coefBG}\\
  \cC(u)[w,b](x) &= \cC_1(u)[w,b](x) + \cC_2(u)[w,b](x) \nonumber \\
  &:= \begin{pmatrix}
    \sigma_+(x, p, u_1(x)) T_\zeta w (p+x)  \\
    \sigma_-(-x, p, u_2(x)) T_\zeta w(p-x)\\
    0
  \end{pmatrix} +  \sigma_* b \overline{\nabla} u
  ,\label{eq:coefCG}
\end{align}
for $u= (u_1,u_2,p)\in \DA,\,b\in\R,\,w\in U:=L^2(\R),\,x\geq 0$. We denote the trace operator by $\cI(u):= (u_1(0), u_2(0))$ and write
\begin{equation}
  \label{eq:1}
  \overline{\nabla} u :=
  \begin{pmatrix}
    \ddx u_1\\ -\ddx u_2 \\ 1
  \end{pmatrix}.
\end{equation}
The constant $c>0$ has to be chosen sufficiently large, as we will see below. Moreover, $\cW:= (W, B)$ is a cylindrical Wiener process on the Hilbert space $\cU := U\oplus \R$, and $\Delta_+$, $\Delta_-$ are the realization of the Laplacian with respective domains
\[\dom(\Delta_\pm) := \{ u \in L^2(\R_+)\,\vert\, \ddx u(0) = \kappa_\pm u(0)\}.\]
The domain of $\cA$ is then given by
\[ \dom(\cA) = \dom(\Delta_+)\times \dom(\Delta_-) \times \R \subset \L^2,\]
which is itself a Hilbert space when equipped with the inner product
\[ \scal{u}{v}{\cA}  := \scal{u}{v}{\L^2} + \scal{\cA u}{\cA v}{\L^2},\quad u,\,v\in \DA.\]
We also introduce the Sobolev spaces $\H^k:= H^k(\R_+)\oplus H^k(\R_+)\oplus \R$, $k\in \N$. With the next theorem we show that under the hypothesis of the previous section, we get at least locally a unique strong solution in the sense of Definition~\ref{df:strongSEE}.
\begin{thm}\label{thm:euSEE}
  Assume that Assumptions~\ref{a:rho}, \ref{a:mu}, \ref{a:sigma} and \ref{a:zeta} hold true and let $c>\max\{\eta_+\kappa_+^2, \eta_-\kappa_-^2\}$. Then, for every $\F_0$-measurable initial data $X_0  \in\H^1$, there exists a unique maximal strong solution $(X,\tau)$ with trajectories almost surely in 
  \[ L^2(0,\tau; \DA) \cap C([0,\tau); \H^1).\]
\end{thm}
\begin{rmk}\label{rmk:isometry}
  To translate the theorem into the SPDE framework of Section~\ref{sec:spde}, we identify $\L^2$ with $L^2(\R)\oplus \R$ using the isometric isomorphism
  \[ \iota: \L^2 \to L^2(\R) \oplus \R,\quad (u_1,u_2,p) \mapsto (u_1\1_{\R_+} + u_2(-(.)) \1_{\R_-}, p).\]  
  In fact, (stochastic) integration can be interchanged with linear continuous operations and we deduce the integral equations for~\eqref{eq:spde} by applying $\iota$ to~\eqref{eq:SEEqG}.
\end{rmk}
\begin{rmk}
  Without much effort, one could replace $\cI$ by any other function which is Lipschitz on bounded sets from $\H^1$ into $\R^n$, and take $\varrho:\R^n\to \R$ locally Lipschitz. In particular, the drift term of $p_*$ in~\eqref{eq:smbp2} might depend also on $p_*$.
\end{rmk}

\subsection{Proof of Theorem~\ref{thm:euSEE}}
\label{ssec:proofeuc}
Theorem~\ref{thm:euSEE} will follow from Theorem~\ref{thm:eucSEEa} for
$p=2$ and $E:= \L^2$, so we just have to verify that Assumption~\ref{a:Aabst},~\ref{a:Babst} and~\ref{a:Cabst} are fulfilled. By diagonal structure, the operator $\cA$ inherits the regularity properties of the Laplacian. 
\begin{lem}\label{lem:linsa}
  $(-\cA,\dom(\cA))$ is positive self-adjoint on $\L^2$.
\end{lem}
Hence, its fractional powers $(-\cA)^\alpha$ can be used to define the inter- and extrapolation spaces, see Appendix~\ref{appendix}, for $\alpha \in \R$,
\begin{equation}
  \label{eq:2}
  E_\alpha := \dom((-\cA)^\alpha), \qquad \norm{u}{\alpha} := \norm{(-\cA)^\alpha u}{\L^2},\; u\in E_\alpha.
\end{equation}
Note that $E_\alpha$ are again Hilbert spaces. We recall the following identities, with equivalence of norms, 
\begin{equation}
  \label{eq:Id}
  E_\alpha =
  \begin{cases}
    \H^{2\alpha},& \alpha \in [0,\sfrac34),\\
    \left\{u \in \H^{2\alpha}\,\vert\, \ddx u_1(0) = \kappa_+ u_1(0),\, \ddx u_2(0) = \kappa_- u_2(0) \right\},& \alpha \in (\sfrac34,1].
  \end{cases}
\end{equation}
\begin{rmk}
  This well-known result was proven by Grisvard in~\cite{grisvardCara}. However, the proof is only given for bounded domains, but works the same (even slightly easier) for half-spaces. A very general versions of this result for half spaces involving also Sobolev and Besov spaces in infinite dimensions is Theorem 4.9.1 in~\cite{amannFunctionSpaces}.
\end{rmk}

\begin{lem}\label{lem:driftlip}
  $\cB: \H^1 \to \L^2$ is Lipschitz continuous on bounded sets. 
\end{lem}
\begin{proof}
  It is well known~\cite{lionsmagenes1} that the trace operator $u\mapsto u(0)$ is linear and continuous from $H^1(\R_+)$ into $\R$. This translates to $\cI$, as a mapping from $\H^1$ into $\R^2$, so that $\varrho\circ \cI$ is Lipschitz on bounded sets from $\H^1$ into $\R$. Moreover, $\overline{\nabla} u$ is clearly Lipschitz from $\H^1$ into $\L^2$ and thus, their product is Lipschitz continuous on bounded sets. 
  
  Let $\mu := \mu_+$ or $\mu:= \mu_-(-(.),.,.,-(.))$. It remains to prove that the Nemytskii operator 
  \[N_\mu(u;p) := \mu(.,p,u(.), \ddx u(.))\]
  is Lipschitz on bounded sets from $H^1(\R_+)\oplus \R$ into $L^2(\R_+)$. Let $u\in H^1(\R_+)$, $p\in \R$, then due to Assumption~\ref{a:mu}.\ref{ai:mugrowths},
  \begin{multline}\label{eq:Nmubd}
    \int_0^\infty \mu(x,p,u(x),\ddx u(x))^2 \d x \leq \\
    \leq 3 \sup_{x\in \R} b(u(x),p)^2 \int_0^\infty  \abs{a(x)}^2 + \abs{u(x)}^2 + \abs{\ddx u(x)}^2 \d x,
  \end{multline}
  which is finite since $u$ is bounded by Sobolev embeddings. Now, for $R>0$ and $u$, $v \in H^1(\R_+)$, $p$, $q\in \R$ such that 
  \[ \norm{u}{H^1},\,\norm{v}{H^1},\,\abs{p},\,\abs{q} \leq R\]
  Assumption~\ref{a:mu}.\ref{ai:mulipu} yields
  \begin{equation*}
    \norm{N_\mu(u;p) - N_\mu(v;p)}{L^2}^2 \leq 3 L_R^2\left( \norm{a_R}{L^2}^2 + 2R^2\right)\norm{u-v}{\infty}^2 + L_R^2 \norm{\ddx u - \ddx v}{L^2}^2.
  \end{equation*}
  On the other hand, by~\ref{a:mu}.\ref{ai:mulipp},
  \begin{multline*}
    \norm{N_\mu(v,p) - N_\mu(v,q)}{L^2}^2  \\ \leq \left(\norm{a_R}{L^2} + \abs{b_R} \left(\norm{u}{L^2} + \norm{\ddx u}{L^2}\right)\right)^2 \abs{p-q}^2\\
      \leq \left(\norm{a_R}{L^2} + 2\abs{b_R} R\right)^2\abs{p-q}^2.
  \end{multline*}
  Combining the latter two equations with~\eqref{eq:Nmubd} we get that $N_\mu$ is Lipschitz on bounded sets for $\mu = \mu_+$ and $\mu = \mu_-(-(.),.,.,.)$, respectively. Here, we again used the Sobolev embedding $H^1\hookrightarrow BUC$. 
\end{proof}
\begin{lem}\label{lem:KatoOneSided}
  Let $\Delta$ be the Laplacian on $L^2(\R_+)$ with domain
  \[ \dom(\Delta):= \setc{ u\in H^2}{ \ddx u(0) = \kappa u(0)},\;\kappa \geq 0.\]
  Then, for all $\eta>0$, $c>0$ and $u\in H^1(\R_+)$ it holds that
  \begin{equation}
    \label{eq:KatoA1}
    \norm{\ddx u}{L^2(\R_+)} \leq \frac1{\sqrt{\eta}} \norm{(c-\eta \Delta)^{\sfrac12} u}{L^2(\R_+)}.
  \end{equation}
  Moreover, if $c>\eta\kappa^2$, then it holds for all $u\in \dom(\Delta)$,
  \begin{equation}
    \label{eq:KatoA2}
    \norm{(c-\eta \Delta)^{\sfrac12} \ddx u}{L^2(\R_+)} \leq \frac1{\sqrt{\eta}} \norm{(c-\eta \Delta) u}{L^2(\R_+)}.
  \end{equation}  
\end{lem}
\begin{rmk}\label{rmk:MaxRegDirichlet}
  In the second statement, it is crucial that $\ddx$ maps $\dom(\Delta)$ into $\dom((c-\eta\Delta)^{\sfrac12})$. This does not hold true anymore for Dirichlet boundary conditions. 
\end{rmk}
\begin{proof}
  \begin{enumerate}[label={Step {\Roman*}:}, fullwidth]
  \item First note that $\dom((c-\Delta)^{\sfrac12}) = H^1(\R_+)$ due to first order boundary conditions. Since $\Delta$ is self adjoint, the same holds true for $(c-\Delta)^{\sfrac12}$, so that for all $u\in \dom(\Delta)$,
    \begin{equation}
      \label{eq:8}
      \norm{(c-\eta \Delta)^{\sfrac12} u}{L^2}^2 = \langle (c-\eta\Delta) u,u\rangle_{L^2} = c \norm{u}{L^2}^2 - \eta \langle \Delta u, u\rangle_{L^2}. 
    \end{equation}
    With integration by parts, we obtain
    \begin{equation}
      \label{eq:A12}
      \norm{(c-\eta\Delta)^{\sfrac12} u}{L^2}^2   = c \norm{u}{L^2}^2 + \eta \norm{\ddx u}{L^2}^2 + \eta\kappa \abs{u(0)}^2.
    \end{equation}
    For the last equality, we just used integration by parts and the fact that $u\in \dom(\Delta)$. Recall that $\dom(\Delta)$ is dense in $\dom((c-\Delta)^{\sfrac12})$, and 
    \[\dom((c-\eta \Delta)^{\sfrac12}) = H^1(\R_+)\hookrightarrow BUC(\R_+),\]
    so that~\eqref{eq:A12} holds true for all $u\in H^1(\R_+)$. 
  \item Now, assume that $c>\eta\kappa^2$. Let $u\in \dom(\Delta)$ and apply~\eqref{eq:A12} to $\ddx u$. If $\kappa>0$, this reads as
    \begin{multline}
      \label{eq:A12D}
      \norm{(c-\eta \Delta)^{\sfrac12} \ddx u}{L^2}^2=  c \norm{\ddx u}{L^2}^2 + \eta \norm{\Delta u}{L^2}^2 + \eta\kappa \abs{\ddx u(0)}^2\\
      = - \scal{(c-\eta\Delta) u}{\Delta u}{L^2} + \left(\eta\kappa - \frac{c}{\kappa}\right)\abs{\ddx u(0)}^2\\
      = \frac1{\eta} \norm{(c-\eta\Delta)u}{L^2}^2 - \frac{c}{\eta} \scal{(c-\eta\Delta)u}{u}{L^2} +  \left(\eta\kappa - \frac{c}{\kappa}\right)\abs{\ddx u(0)}^2.
    \end{multline}
    The second equality follows again by integration by parts. 
    We get~\eqref{eq:KatoA2}, since $-\Delta$ is non-negative self-adjoint and $u\in \dom(A)$. 

    For $\kappa= 0$,~\eqref{eq:A12} still holds true and the computation in~\eqref{eq:A12D} reduces to,
    \begin{equation}
      \norm{(c-\eta \Delta)^{\sfrac12} \ddx u}{L^2}^2 =- \scal{(c-\eta\Delta) u}{\Delta u}{L^2}  \leq \frac1{\eta} \norm{(c-\eta\Delta)u}{L^2}^2. \qedhere
    \end{equation}
  \end{enumerate}    
\end{proof}

A direct consequence of fundamental theorem of calculus is, see also \cite[Proof of Lemma 5.1]{sowersEtAl} or \cite[Appendix A]{lunardiInterpol},
\[ u(x) = \int_{x}^{x+1} u(y) \d y - \int_{x}^{x+1} (x+1-y) \nabla u(y) \d y,\]
for all $u\in H^1(\R_+)$, $x>0$, which yields
\begin{equation}
  \sup_{x>0}\abs{u(x)} \leq \sqrt{2} \norm{u}{H^1}.
\end{equation}
Thus, the equality in~\eqref{eq:A12} implies the following two-sided estimate.
\begin{cor}
  With the notation of Lemma~\ref{lem:KatoOneSided}, for all $u\in H^1(\R_+)$,
  \[\sqrt{c\wedge \eta} \norm{u}{H^1} \leq \norm{(c-\eta \Delta)^{\sfrac12}u}{L^2} \leq \left(\sqrt{c\vee \eta} + \sqrt{2\eta\kappa}\right)\norm{u}{H^1}.\]
\end{cor}
\begin{rmk}\label{rmk:katosquareroot}
  When replacing $\eta \Delta$ by a general uniformly elliptic operator of second order $A$, we still know that there exists constants $K_0$, $K_1$ such that
  \[ K_0 \norm{u}{H^1} \leq \norm{(-A)^{\sfrac12}u}{L^2} \leq K_1 \norm{u}{H^1}\]
  for all $u\in \dom((-A)^{\sfrac12})$. In fact, this question is known as Kato's square root problem and was solved by Auscher et al, see \cite[Thm 6.1]{katosqroot}. This theory is strongly based on the bounded $H^\infty$-calculus of $-A$. However, as can be seen in Appendix~\ref{appendix}, just to know plain existence of such constants without exact bounds might not be sufficient for the discussion of existence for stochastic evolution equations.
\end{rmk}
\begin{lem}\label{lem:noiselip}$\;$
  \begin{enumerate}[label=(\roman*)]
  \item $\cC_1: \H^1 \to \HS(\cU; \H^1)$ is Lipschitz continuous on bounded sets.
  \item $\cC_2$ is Lipschitz continuous from $\dom(\cA)$ into $\HS(\cU; E_{\frac12})$. More precisely, there exists $L_*<\sqrt2$ such that for all $u$, $v\in \dom(\cA)$,
    \begin{equation}
      \label{eq:LipNormC2}
      \norm{\cC_2(u) - \cC_2(v)}{\HS(\cU;E_{\sfrac12})} \leq L_*\norm{-\cA (u-v)}{\L^2}.
    \end{equation}
  \end{enumerate}
\end{lem}
\begin{proof}
  For the first part, we use the results in \cite[Appendix A]{SFBP}, to get that $u\mapsto N_\sigma(u;p) := \sigma(.,p,u(.))$ is Lipschitz on bounded sets on $H^1(\R_+)$, for each $p\in \R$. 
  Here, we let $\sigma:= \sigma_+$ or $\sigma := \sigma_-(-.,.,.)$. Moreover, the weak derivative of $N_\sigma(u;p)$ is
  \[ \tfrac{\d}{\d x}  N_\sigma(u;p)(x) = \ddx \sigma(x,p,u(x)) + \ddy \sigma(x,p,u(x)) \ddx u(x),\]
  and by Assumption~\ref{a:sigma},
  \begin{multline*}
    \norm{\ddy \sigma(.,p,u(.))u(.) - \ddy \sigma(.,\tilde p,u(.))u(.)}{L^2}\leq  \\ \leq \norm{u}{L^2} \sup_{x\geq 0}\abs{\ddy \sigma(x,p,u(x)) - \ddy \sigma(x,\tilde p, u(x))} 
    \leq  \norm{u}{L^2} \tilde b_R \abs{p-\tilde p},
  \end{multline*}
  where $R>0$ is such that $\max\{\norm{u}{\infty}, p, \tilde p\}\leq R$. Moreover,
  \begin{equation*}
    \norm{N_\sigma(u;p) - N_\sigma(u;\tilde p)}{L^2} \leq b_R \norm{a_R + \abs{u}}{L^2} \abs{p-\tilde p},
  \end{equation*}
  and the same estimates are valid when replacing $\sigma$ by $\ddx \sigma$. Hence, we get that $p\mapsto N_{\sigma}(u;p)$ is Lipschitz on bounded sets from $\R$ into $H^1(\R_+)$, so that $(u,p)\mapsto N_\sigma(u;p)$ is Lipschitz on bounded sets from $H^1(\R_+) \oplus \R$ into $H^1(\R_+)$.
  
  From the proof of Lemma B.4 in~\cite{SFBP} we extract the estimate 
  \begin{multline}
    \label{eq:HSestH1}
    \norm{N_\sigma(u;p) T_\zeta}{\HS(L^2(\R); H^1)}\leq \\
    \leq K \norm{N_\sigma(u;p)}{H^1} \sup_{z\in \R}\left(\norm{\zeta(z,.)}{L^2(\R)}+\norm{\ddx \zeta(z,.)}{L^2(\R)}\right).
  \end{multline}
  Writing $\zeta_{y}(x,.) := \zeta(x+y,.) - \zeta(x,.)$, we get for all $w\in L^2(\R)$, $x$, $y$, $z\in \R$,
  \[T_\zeta w(x+y) - T_\zeta w(x+z) = T_{\zeta_{y-z}} w(x),\quad \forall w\in L^2(\R),\;x\in \R.\]
  Using Assumption~\ref{a:zeta} and fundamental theorem of calculus, one shows that 
  \begin{align*}
    \sup_{z\in \R} \norm{\zeta_x(z,.)}{L^2(\R)} &\leq \abs{x} \sup_{z\in \R} \norm{\ddx \zeta(z,.)}{L^2(\R)},\\
    \sup_{z\in \R} \norm{\ddx \zeta_x(z,.)}{L^2(\R)} &\leq \abs{x} \sup_{z\in \R} \norm{\ddxx \zeta(z,.)}{L^2(\R)}.
  \end{align*}
  See also~\cite[Lemma B.2]{SFBP} for details. Combining the latter three equations with the first part on $N_\sigma$, we get that $\cC_1$ is, indeed, Lipschitz on bounded sets.
  
  To prove the second part of the lemma, note that for any CONS $(e_k)_{k\in\N}$ of $L^2(\R)$ the family $((0,1), (e_1,0),(e_2,0),...)$ is a CONS of $\cU$. Hence,
  \begin{equation*}
    \norm{\cC_2(u) - \cC_2(v)}{\HS(\cU;E_{\frac12})}^2 = \norm{\cC_2(u)[0,1]- \cC_2(v)[0,1]}{\frac12}^2 .
  \end{equation*}
  By diagonal structure of $\cA$ and $\L^2$ we have
  \[(-\cA)^{\sfrac12} =
  \begin{pmatrix}
    (c - (\eta_++\tfrac12 \sigma_*^2)\Delta_1)^{\frac12} &0&0\\
    0& (c - (\eta_-+\tfrac12 \sigma_*^2)\Delta_2)^{\frac12}&0\\
    0&0&c^{\frac12}
  \end{pmatrix},\]
  so that the second part of Lemma~\ref{lem:KatoOneSided} yields
  \begin{equation*}
    \norm{\cC_2(u) - \cC_2(v)}{\HS(\cU;E_{\sfrac12})} \leq L_* \norm{(-\cA) (u-v)}{\L^2},
  \end{equation*}
  for $L_* := \sigma_* \left((\eta_+\wedge \eta_-) + \tfrac12 \sigma_*^2\right)^{-\frac12} <\sqrt2$.
\end{proof}
Putting things together we get that Assumptions~\ref{a:Aabst},
\ref{a:Babst}, and~\ref{a:Cabst} are fulfilled. Moreover,
\eqref{eq:LipNormC2} and Remark~\ref{rmk:para} show that also
\eqref{eq:para} holds true with $L_B = 0$, $L_C = L_*$ and $p=2$. Thus, application of
Theorem~\ref{thm:eucSEEa} finishes the proof of
Theorem~\ref{thm:euSEE}. \qed

\subsection{Explosion times}
We now formulate and prove Theorem~\ref{thm:tau0tau} and \ref{thm:global} in the framework of stochastic evolution equations on $\L^2$, see also Remark~\ref{rmk:isometry}.
\begin{thm}\label{thm:globalSEE}
  Let $(X,\tau)$ be the unique maximal solution of~\eqref{eq:SEEqG} on $\L^2$ and assume that, in addition to the assumptions of Theorem~\ref{thm:euSEE} also Assumptions~\ref{a:musigmalg} holds true and $\varrho$ is globally bounded. Then, $\tau = \infty$ almost surely. If, moreover, $X_0\in L^{2}(\Omega; \H^1)$, then for all $T>0$ there exists a constant $K>0$ such that
  \begin{align*} \EE\int_0^T \norm{X(s)}{\cA}^2 \d s + \E{\sup_{0\leq t\leq T} \norm{X(s)}{\H^1}^2} \leq K \left(1 + \E{\norm{X_0}{\H^1}^2}\right).\end{align*}
\end{thm}
\begin{proof}
  First, since $\varrho$ is globally bounded, we get that
  \[ (u_1,u_2,p) \mapsto \varrho(\cI(u)) \overline{\nabla} u\]
  has linear growths as a map from $\H^1$ into $\L^2$. From~\eqref{eq:Nmubd} we get the linear growths bound, for all $u\in H^1(\R_+)$, $p\in \R$,
  \[ \int_0^\infty \abs{\mu_{\pm}(\pm x,p, u(x),\ddx u(x))}^2 \d x \leq  3 \norm{b}{L^\infty(\R^2)} \left( \norm{a}{L^2}^2 +\norm{u}{H^1}^2\right).\] 
  Hence, the Nemytskii-operator $N_\mu$, defined in the proof of Lemma~\ref{lem:driftlip}, has linear growths from $H^1\oplus \R$ into $L^2$ and $\cB$ has linear growths from $\H^1$ into $\L^2$. 

  With the same arguments, we get linear growths of $N_\sigma$ from $H^1\oplus \R$ into $L^2$, for $\sigma\in \{\sigma_+,\ddx \sigma_+,\sigma_-, \ddx \sigma_-\}$. Finally, for all $p\in \R$, $u\in H^1(\R_+)$, $x>0$, it holds that
  \[ \abs{\ddy \sigma(x,p,u(x)) \ddx u(x)} \leq \lVert \tilde b \rVert_{L^\infty(\R^2)} \abs{\ddx u(x)}.\]
  The weak derivative of $N_\sigma(u)$ is given by $\ddx \sigma(x,p,u(x)) + \ddy \sigma(x,p,u(x))\ddx u(x)$, and therefore $N_\sigma$ has linear growths from $H^1\oplus \R$ into $H^1$. Using estimate~\eqref{eq:HSestH1} and the structure of $\cC_1$, we refer to the proof of part (i) of Lemma~\ref{lem:noiselip} for details, we observe linear growths for
  \[\cC_1 : \H^1 \to \HS(\cU;\H^1).\]
  Summarizing, the assumptions of Theorem~\ref{thm:globalSEEabst} with $q= 2$ are fulfilled and Theorem~\ref{thm:globalSEE} follows.
\end{proof}
\begin{proof}[Proof of Theorem~\ref{thm:tau0tau}]
  The proof works similar to the proof of~\cite[Theorem 4.5]{SFBP}. 
  For $N\in \N$ let $\varrho_N:\R^2\to \R$ be a locally Lipschitz continuous function such that
  \[ \varrho_N(x,y) =
  \begin{cases}
    \varrho(x,y),& \abs{(x,y)}\leq N,\\
    0,& \abs{(x,y)} > N+1.
  \end{cases}
  \]
  In consistency with section~\ref{sec:spde}, set
  \begin{equation}
    \label{eq:17}
    \tau_0^N := \inf\left\{t\geq 0\,\mid\, t<\tau, \abs{\cI(X(t))} >N\right\},\qquad \tau_0 := \lim_{N\to\infty} \tau_0^N,
  \end{equation}
  using the convention $\inf \emptyset = \infty$.
  By continuity of the trace operator, we have 
  \[ N \leq  \abs{\cI(X(\tau_0^N))} \leq K_{\cI} \norm{X(\tau_0^N)}{\H^1},\quad\text{on }\{\tau_0<\infty\}.\]
  In particular, on $\{\tau_0<\infty\}$,
  \[\lim_{N\to\infty} \norm{X(\tau_0^N)}{\H^1} = \infty,\]
  which yields, due to $\H^1$-continuity of $X$, that $\tau_0\geq \tau$ almost surely. 
  
  On the other side, replacing $\varrho$ by $\varrho_N$, the stochastic evolution equation~\eqref{eq:SEEqG} admits a unique global solution $X_N$ by Theorem~\ref{thm:globalSEE}. By definition of $\varrho_N$, $(X_N, \tau_0^N)$ is a local solution of the original equation, so that the uniqueness claim of Theorem~\ref{thm:euSEE} yields $X= X_N$ on $\llbrak 0,\tau_0^N\llbrak$, and $\tau_0^N \leq \tau$ almost surely, for all $N\in \N$.
\end{proof}

\section{Distributional Solutions and Transformation}
\label{sec:trafo}
The transformation from the fixed to the moving boundary problem will be performed by It\={o}-Wentzell formula, in its version proven by Krylov~\cite{krylovItoWentzell}. To this end, we first have to rewrite the SPDEs considered above into an equation on the distribution space. Recall that the cylindrical $\Id$-Wiener process $W$ on $U=L^2(\R)$ can be written as
\[ W_t = \sum_{k=1}^\infty e_k \beta_t^k,\qquad t\geq 0,\]
for an orthonormal basis $(e_k)_{k\in\N}$ of $U$, and independent Brownian motions $\beta^k$, $k\geq 1$. For consistent notation, we set $\beta^0 := B$.

We denote by $C^\infty_0= C^\infty_0(\R)$ the space of smooth real functions with compact support and $\sD$ the space of distributions. Let $\ell^2$ be the space of real square summable sequences. We denote by $\sD(\ell^2)$ the space of $\ell^2$-valued distributions on $C^\infty_0$. That is, linear $\ell^2$-valued functionals such that $\phi \mapsto \langle g,\phi\rangle = (\langle g^k,\phi \rangle)_k$ is continuous with respect to the standard convergence of test functions. 

For a predictable stopping time $\tau$ there exists an announcing sequence $(\tau_n)_{n\in \N}$. That is, $\lim_{n\to\infty} \tau_n = \tau$ almost surely and $\tau_n < \tau$ on $\{\tau>0\}$, cf.~\cite[Rmk 2.16]{js}. We will use this notation in this section without further mentioning.

\begin{rmk}\label{rmk:HSl2}
  Note that $\HS(U;E)$ is isometric isomorphic to $\ell^2(E)$. More precisely, any element $T\in \HS(U;E)$ can be identified with $(Te_k)_{k\in \N}\in \ell^2(E)$, for a CONS $(e_k)_{k\in \N}$ of $U$.  
\end{rmk}
\begin{lem}\label{lem:HSl2imb}
  Let $U$, $\tilde E$ and $E$ be separable Hilbert spaces with $\tilde E \hookrightarrow E$. Then,
  \[\HS (U;\tilde E) \hookrightarrow \HS(U;E),\quad \text{and}\quad \ell^2(\tilde E) \hookrightarrow \ell^2(E).\]
\end{lem}
\begin{proof}
  First, it is clear that for $(g^k)_{k\in \N}\in \ell^2(\tilde E)$ it holds that 
  \[ \sum_{k=1}^\infty \norm{g^k}{E}^2  \leq K^2 \sum_{k=1}^\infty \norm{g^k}{\tilde E}^2.\]

  Hence, we can consider $\ell^2(\tilde E)$ as a subset of $\ell^2(E)$. One can also show that from density of $\tilde E$ in $E$ it follows density of $\ell^2(\tilde E)$ in $\ell^2(E)$, but we skip the details here. By identification, the results also hold true for $\HS$. 
\end{proof}
From strong continuity of the shift group on $L^2(\R)$ we get the following basic result.
\begin{lem}\label{lem:shift}
  For all $x\in \R$, the shift operation $u\mapsto u(.+x)$ is isometric isomorphic from $H^k(\dot \R)$ into $H^k(\R\setminus \{x\})$, for all $k\geq 0$. In addition, the operation 
  \[ (u,x) \mapsto u(.+x),\]
  is continuous from $L^2(\R) \oplus \R$ into $L^2(\R)$.
\end{lem}

Using the Riesz isomorphism we consider $L^2(\R)$ as a subset of $\sD$. The following lemma is the corresponding result addressing $\ell^2$.
\begin{lem}\label{lem:L2distr}
  By identification,
  \[L^2(\R)\subset \sD,\qquad \text{and}\qquad \ell^2(L^2(\R)) \subset \sD(\ell^2).\] 
\end{lem}
\begin{proof}
  Let $h = (h^k)_{k\geq 0}\in \ell^2(L^2)$, then
  \[ \phi \mapsto (\langle h,\phi \rangle^k)_{k\geq 0}  := (\langle h^k, \phi \rangle_{L^2})_{k\geq 0}\]
  defines a continuous linear $\ell^2$-valued function on $C^\infty_0$, since Cauchy-Schwartz inequality yields
  \begin{equation}
    \label{eq:distCSI}
    \sum_{k=0}^\infty \abs{\langle h^k,\phi \rangle }^2 \leq \norm{h}{\ell^2(L^2)}^2 \norm{\phi}{L^2}^2,\qquad \forall \phi \in C^\infty_0.   \qedhere 
  \end{equation}
\end{proof}

\subsection{It\={o}-Wentzell Formula}
\begin{defn}
  A $\sD$-valued stochastic process $f=(f_t)$ is called \emph{predictable}, if for all $\phi \in C^\infty_0$, the real valued stochastic process $(\langle f_t,\phi\rangle)_t$ is predictable. In the same way, we call a $\sD(\ell^2)$-valued stochastic process $g= (g_t)$ predictable, if $(\langle g_t,\phi \rangle)_t$ is $\ell^2$-predictable for all $\phi \in C^\infty_0$. 
\end{defn}
We are now interested in equations on $\sD$, of the form
\begin{equation}
  \label{eq:spdeD}
  \d u(t,x) = f_t(x) \d t + \sum_{k=0}^\infty g^k_t(x) \d \beta^k_t,
\end{equation}
with initial conditions $u(0,x) = u_0(x)$, which are assumed to be $\sD$-valued and $\F_0$-measurable.
\begin{ass}
  \label{a:distr} 
  \begin{enumerate}[label=(\roman*)]
  \item\label{ai:Dint} For a $\sD$-valued predictable process $f$ assume that for all $\phi \in C^\infty_0$ and all $R$, $T>0$ it holds that
    \[ \int_0^T \sup_{\abs{x}\leq R} \abs{\langle f_t,\phi(.-x)\rangle } \d t <\infty \quad\PP\text{-almost surely}.\]
  \item\label{ai:Dl2int} For a $\sD(\ell^2)$-valued predictable process $g$ assume that for all $\phi \in C^\infty_0$ and all $R$, $T>0$ it holds that
    \[ \int_0^T \sup_{\abs{x}\leq R} \norm{\langle g_t,\phi(.-x)\rangle }{\ell^2}^2 \d t <\infty \quad\PP\text{-almost surely}.\]
  \end{enumerate}
\end{ass}
\begin{rmk}
  If a $\sD(\ell^2)$-valued stochastic process $g$ satisfies part (ii) of Assumption~\ref{a:distr}, then for all $T>0$, $\phi \in C^\infty_0$, almost surely
  \begin{equation}
    \label{eq:16}
    \sum_{k=0}^\infty \int_0^T \langle g_t^k, \phi\rangle^2 \d t = \int_0^T \norm{\langle g_t,\phi \rangle}{\ell^2}^2 \d t.
  \end{equation}
\end{rmk}

\begin{defn}\label{df:distrSPDE}
  Let $f$ and $g$ be predictable processes on $\sD$ and $\sD(\ell^2)$, respectively, and $\tau$ be a predictable stopping time such that $f$ and $g$ satisfy Assumption~\ref{a:distr} on $\llbrak 0,\tau_n\rrbrak$ for all $n\in\N$. Then, a $\sD$-valued predictable process is called \emph{(local) solution in the sense of distributions} of~\eqref{eq:spdeD}, if for all $\phi \in C^\infty_0$ it holds on $\llbrak 0,\tau \llbrak$,
  \[ \langle u(t), \phi \rangle  = \langle u_0,\phi\rangle + \int_0^t \langle f_s ,\phi \rangle \d s + \sum_{k=0}^\infty \int_0^t \langle g_s^k, \phi \rangle \d \beta^k_s.\]
\end{defn}
Our aim is to shift the solutions of SPDEs by a one dimensional It\={o}-diffusion. On its coefficients, we impose the following conditions. 
\begin{ass}  \label{a:xt}
  Assume that the real predictable processes $b= (b_t)_{t\geq 0}$ and $\nu^k = (\nu^k_t)_{t\geq 0}$, $k\in \N_0$ satisfy for all $t>0$ almost surely
  \[ \int_0^t \abs{b_s} + \norm{(\nu^k_s)_{k\in \N_0}}{\ell^2}^2 \d s < \infty.\]
\end{ass}
The next theorem is a version of \cite[Theorem 1.1]{krylovItoWentzell} reformulated for processes which exist up to predictable stopping times. In fact, when $\tau$ is a predictable stopping time we can apply Krylov's result on $\llbrak 0,\tau_n \rrbrak$ for all $n\in\N$. We denote the first two distributional derivatives on $\R$ by $\Ddx$ and $\Ddxx$, respectively. 
\begin{thm}[It\={o}-Wentzell Formula]\label{thm:itowentzell}
  Let $\tau$ be a predictable stopping time with announcing sequence $(\tau_n)_{n\in\N}$. Moreover, let $f$ and $g$ be resp. $\sD$- and $\sD(\ell^2)$-predictable processes such that $f$ and $g$ satisfy Assumption~\ref{a:distr} on $\llbrak 0,\tau_n\rrbrak$, for all $n\in\N$ and $u$ is a local distributional solution on $\llbrak 0,\tau\llbrak$ of
  \[\d u_t(x) = f_t(x) \d t + \sum_{k=0}^\infty g^k_t(x) \d \beta^k.\]
  Moreover, consider real predictable processes $b = (b_t)$, $(\nu_t^k)_{t\geq 0}$, $k\in \N_0$ such that, on $\llbrak 0,\tau_n\rrbrak$, they satisfy Assumption~\ref{a:xt}. Let $x_t$ be given on $\llbrak 0,\tau\llbrak$ by
  \[\d x_t = b_t \d t + \sum_{k=0}^\infty \nu^k_t \d \beta^k_t.\]
  Then, $v_t(x):= u_t(x+x_t)$ is a local distributional solution on $\llbrak 0,\tau\llbrak$ of
  \begin{multline*}
    \d v_t(x) = \left[ \tfrac12 \sum_{k=0}^\infty \abs{\nu_t^k}^2  \Ddxx v_t(x) + b_t \Ddx v_t(x) + \sum_{k=0}^\infty \Ddx g_t^k(x+x_t) \nu^k_t \right] \d t\\
    +f_t(x+x_t) \d t + \sum_{k=0}^\infty \left[g_t^k(x+ x_t)+  \Ddx v_t(x) \nu^k_t\right] \d\beta^k_t.
  \end{multline*}
\end{thm}

\subsection{Proof of Theorem~\ref{thm:EUsmbp}}

Let $(X,\tau)$ be the unique maximal strong solution of \eqref{eq:SEEqG} on $\L^2$ and $(\tau_n)_{n\in \N}$ and announcing sequence for $\tau$. Define the isometry, see Remark~\ref{rmk:isometry},
\[ \iota :  \L^2 \to L^2(\R)\oplus \R ,\quad (u_1,u_2,p) \mapsto (u_1\1_{\R_+} + u_2(-(.)) \1_{\R_-},p)\]
and set  $(u(t,.),p_*(t)) := \iota X(t)$ on $\llbrak 0,\tau\llbrak$. To recover the notation of~\eqref{eq:spdeD}, write
\begin{align*}
  f_t(x) &:=\bar \mu(x,p_*(t), u(t,x),\nabla u(t,x), \Delta u(t,x))\\
         &\qquad \qquad+ \frac{\sigma_*^2}2 \Delta u(t,x) + \varrho(\cI(u(t)))\nabla u(t,x),\\
  g_t^k(x) &:= \bar\sigma(x, p_*(t), u(t,x)) T_{\zeta} e_k(x),\quad k\geq 1,
\end{align*}
and $g_t^0(x) := \sigma_* \nabla u(t,x)$. Recall that $\nabla u$ and $\Delta u$ denote the first two piece-wise weak derivatives, which are assumed to exist as elements in $L^2(\R)$. The functions $\bar\mu$ and $\bar\sigma$ have been defined resp. in~\eqref{eq:bmu} and~\eqref{eq:bsigma}. Note that on $\llbrak 0,\tau\llbrak$
\[ \iota \cC(X(t))[e_k,0](x) = (g_t^k(x),0) ,\quad \text{and}\quad\iota\cC(X(t))[0,1](x) = (g_t^0(x),\sigma_*).\] 
Obviously, $\iota$ is also isometric isomorphic from $\H^\alpha$ into $H^\alpha (\dot \R)\oplus \R$ for all $\alpha >0$. Recall that $\cC(X(t))\in \HS(\cU; \H^1)$ and thus $\iota \cC(X(t)) \in \HS(\cU; H^1(\dot \R)\oplus \R)$ almost surely. By Lemma~\ref{lem:HSl2imb} in combination with Remark~\ref{rmk:HSl2} we get that $(g_t^k)_{k\geq 0}$ is $\ell^2(L^2(\R))$-continuous on $\llbrak 0,\tau \llbrak$. Localizing up to $\tau_n$, for each $n\in\N$, we also obtain (square) integrability on $f_t$ and $g_t$ by Cauchy-Schwartz inequality and~\eqref{eq:distCSI}, respectively. Moreover, Lemma~\ref{lem:HSl2imb} yields
\begin{equation*}
  \int_0^t \sup_{\abs{x}\leq R} \norm{\scal{g_s}{\phi(.-x)}{}}{\ell^2}^2 \d s \leq \norm{\phi}{L^2(\R)}^2\int_0^t \norm{\cC(X(s))}{\HS(\cU;\L^2)}^2 \d s,
\end{equation*}
which is finite on $\llbrak 0,\tau_n \rrbrak$, for all $n\in \N$ and $R>0$. Indeed, since $\tau_n <\tau$ on $\{\tau>0\}$, $X(.\wedge \tau_n)$ has paths in $C([0,\tau_n];\H^1)$ almost surely. By Lemma~\ref{lem:noiselip}, $\cC$ is continuous from $\H^1$ into $\L^2$ which yields the integrability property of $(g_t)$. Choosing $x_t :=  -p_*(t)$, we get the remaining integrability claims in a similar way and obtain that all assumptions of Theorem~\ref{thm:itowentzell} are fulfilled. 

Since testing against test functions is a continuous linear operation on $L^2(\R)$, $u$ is also a solution in the sense of Definition~\ref{df:distrSPDE}. For $v(t,x) := u(t,x - p_*(t))$ we get by Theorem~\ref{thm:itowentzell},
\begin{multline}
  \d v(t,x) = f_t(x+x_t) \d t + \\
  + \left[\frac{\sigma_*^2}2 \Ddxx v(t,x) - \varrho(\cI(u(t,.))) \Ddx v(t,x) - \sigma_*\Ddx g_t^0(x+x_t) \right]\d t + \\
  + \sum_{k=1}^\infty g^k_t (x+x_t) \d \beta^k_t +
  \left[g_t^0(x+x_t) -  \sigma_*\Ddx v(t,x)\right] \d\beta^0_t.\label{eq:vC1}
\end{multline}
Note that 
\begin{align*}
  f_t(x+x_t) &= \bar \mu(x-p_*(t),p_*(t),v(t,x), \nabla v(t,x), \Delta v(t,x)) + \frac{\sigma_*}2 \Delta v(t,x) \\
             &\qquad\qquad\qquad\qquad\qquad\qquad+ \varrho(v(t,p_*(t)+), v(t,p_*(t)-)) \nabla v(t,x),\\
  g^k_t(x+x_t) &= \bar \sigma(x-p_*(t),p_*(t),v(t,x)) T_\zeta e_k(x),\quad k\geq 1,
\end{align*}
and hence,
\begin{align*}
  \d v(t,x) &= \left[ \vphantom{\tfrac{\sigma^2_*}2}\bar \mu(x-p_*(t),p_*(t), v(t,x), \nabla v(t,x), \Delta v(t,x)) \right. \\
            &\quad\quad + \varrho(v(t,p_*(t)+), v(t,p_*(t)-)) \left( \nabla v(t,x) - \Ddx v(t,x) \right) \\
            &\quad\quad  \left. + \frac{\sigma_*^2}2  \left( \Delta v(t,x) + \Ddxx v(t,x) - 2 \Ddx \nabla v(t,x)\right) \right] \d t\\
            &\quad + \bar \sigma(x-p_*(t),p_*(t), v(t,x)) \d \xi_t(x)  + \sigma_*\left( \nabla v(t,x) - \Ddx v(t,x)\right) \d\beta^0_t.
\end{align*}
Moreover, for $h\in H^1(\R\setminus\{p\})$, $p\in \R$, it holds that
\begin{equation*}
  \Ddx h - \nabla h = (h(p+) - h(p-)) \delta_p,
\end{equation*}
where $\delta_p$ is the Dirac distribution with mass at $p$. This indeed holds true for all $h\in H^1(\R\setminus\{p\}) \cap BUC^1(\R\setminus\{p\})$ and then extends by density to all of $H^1(\R\setminus\{p\})$. Inserting into~\eqref{eq:vC1} yields
\begin{align*}
  \d v(t,x) &=  \bar \mu(x-p_*(t),p_*(t) v(t,x) ,\nabla v(t,x) ,\Delta v(t,x)) \d t\\
            &\qquad + \bar \sigma(x-p_*(t),p_*(t),v(t,x))\d\xi_t(x)\\
            &\qquad + L_1(v(t,.), p_*(t)) \d p_*(t) + \tfrac12 L_2(v(t,.),p_*(t)) \d\,[p_*](t),
\end{align*}
where $L_1$ and $L_2$ have been defined in Section~\ref{sec:spde} as
\begin{align*}
  L_1(v,p) &= -\left( v(p+) - v(p-)\right) \delta_p, \\
  L_2(v,p) &= \left(v(p+) - v(p-)\right) \delta_p' - \left( \nabla v(p+) - \nabla v(p-)\right) \delta_p.
\end{align*}
Here, $\delta'_p$ is the distributional derivative of $\delta_p$. Finally, we use that the shift operation $(h,x) \mapsto h(.-x)$ is continuous on $L^2(\R)\oplus \R$ and hence, the paths of $v$, $\nabla v$, $\Delta v$ and $p_*$ inherit the space-time-regularity which is claimed in Theorem~\ref{thm:EUsmbp}.\\

To show uniqueness, let $(w, q_*,\varsigma)$ be another local solution such that, as functions of time, $q_*$, $w$ and $\nabla w$ are continuous and $\Delta w$ is square integrable on $L^2(\R)$. Let  $(\varsigma^n)$ be an announcing sequence for $\varsigma$ and set $u(t,.) := w(t,.+q_*(t))$. From Lemma~\ref{lem:shift}, we get almost surely,
\[u\in C([0,\varsigma); H^1(\dot \R))\cap L^2([0,\varsigma);H^2(\dot \R)).\]
By definition of $\Gamma^2(x)$, $x\in \R$, we get in addition that $u(t,.)$ fulfills the boundary conditions~\eqref{eq:bcc} for almost all $t>0$. Now, it suffices to show that $(u,\varsigma)$ is a local strong solution of~\eqref{eq:spde}. To this end, set
\begin{align*}
  f_t(x) &:= \bar\mu(x-q_*(t),q_*(t),w(t,x),\nabla w(t,x),\Delta w(t,x)) \\
         &\qquad +\varrho(w(t,q_*(t)+),w(t,q_*(t)-))L_1(w(t,.), q_*(t)) \\
         &\qquad +  \tfrac12 \sigma_* L_2(w(t,.),q_*(t)),\\
  g_t^k(x) &:= \bar \sigma(x-q_*(t),q_*(t), w(t,x)) T_\zeta e_k(x),\quad k\geq 1,
\end{align*}
and $g_t^0(x) := 0$. 

Since we know already that $u$ has the path regularity which was asked
for, we can again apply the procedure from the existence part of this
proof to get the measurability and integrability properties of
$g_t^k$, $k\geq 1$, and the the part of $f_t$ involving $\bar\mu$.

Moreover, for $\phi \in C_0^\infty $ and $R>0$, we get on $\llbrak 0,\varsigma_n\rrbrak$,
\begin{multline*}
  \int_0^t \sup_{\abs{x}<R}\abs{\scal{L_1(w(s,.), q_*(s))}{\phi(.-x)}{}}^2\d s \leq \\
  \leq  \norm{\phi}{\infty}^2 \int_0^t \abs{u(s,0+) - u(s,0-)}^2 \d s,
\end{multline*}
and, similarly,
\begin{multline*}
  \int_0^t\sup_{\abs{x}<R} \abs{\scal{L_2(w(s,.),q_*(s))}{\phi(.-x)}{}}\d s \leq  \\ 
  \leq \norm{\phi}{C^1} \left(\int_0^t \abs{u(s,0+) - u(s,0-)}\d s+
    +  \int_0^t \abs{\nabla u(s,0+) - \nabla u(s,0-)} \d s\right).
\end{multline*}

By continuity of the trace operator on $H^1(\dot \R)$, and since almost surely $u\in L^2([0,\varsigma); H^2(\dot \R))$, we get that all these integrals are finite on $\llbrak 0,\varsigma_n\rrbrak$, for all $n\in \N$, and the assumptions of Theorem~\ref{thm:itowentzell} are fulfilled for $w$ and $q_*$. Hence, It\={o}-Wentzell formula yields that $u$ is a solution, in the sense of distributions, of~\eqref{eq:spde}. 

To get to the notion of strong solutions, we switch back to the framework of Section~\ref{sec:centered}. Set $Y(t):= \iota^{-1}(u(t),q_*(t))$ on $\llbrak 0,\varsigma \llbrak$. Since $\iota$ is isometric isomorphic, we can write $Y$ on $\llbrak 0,\varsigma\llbrak$ as
\begin{multline}
  \label{eq:Yweak}
  \scal{Y(t)}{\phi}{\L^2}  - \scal{Y(0)}{\phi}{\L^2} \\
  =  \int_0^t \scal{\cA Y(s)}{\phi}{\L^2} + \scal{\cB(Y(s))}{\phi}{\L^2} \d s + \int_0^t \scal{\cC(Y(s))\d \cW_s}{\phi}{\L^2},
\end{multline}
 for all $\phi = (\phi_1,\phi_2,\phi_3) \in C^{\infty}_0(\R_+) \times C^{\infty}_0(\R_+)\times \R$. By assumptions on $w$ and $q_*$ we get that $Y$ has paths in $L^2(0,\varsigma;\DA)$ and in $C([0,\varsigma[; \H^1)$ almost surely. We recall from Section~\ref{sec:centered}, that $\cB:\H^1\to \L^2$ and $\cC_1: \H^1\to \HS(U;\H^1)$ are Lipschitz continuous on bounded sets, which yields that $\cB\circ Y$ is $\L^2$-continuous and $\cC_1\circ Y$ is $\HS(U; \L^2)$-continuous. In particular, both are locally bounded, so that on $\llbrak 0,\varsigma\llbrak$,
\begin{equation*}
  \int_0^t \norm{\cA Y(s)}{\L^2} + \norm{\cB(Y(s))}{\L^2}  + \norm{\cC_1(Y(s))}{\HS(U;\L^2)}^2 \d s <\infty,
\end{equation*}
Due to global Lipschitz continuity of $\cC_2$ there exists $K>0$ such that on $\llbrak 0,\varsigma\llbrak$,
\[\int_0^t \norm{\cC_2(Y(s))}{\HS(\R;\L^2)}^2\d s \leq  K\int_0^t 1 + \norm{Y(s)}{\cA}^2\d s < \infty.\] 
Consequently, the terms involved are respectively Bochner and stochastically integrable on $\L^2$ and we can interchange the inner product of $\L^2$ with integration in~\eqref{eq:Yweak}. By density of $C^\infty_0(\R_+)$ in $L^2(\R_+)$, the equation holds for all $\phi \in \L^2$, and thus the strong integral equation holds true, i.\,e.
\[ Y(t) = Y(0) +  \int_0^t \cA Y(s) + \cB(Y(s)) \d s + \int_0^t \cC(Y(s))\d \cW_s,\quad\text{on } \llbrak 0,\varsigma\llbrak.\]
The uniqueness and maximality part of Theorem~\ref{thm:euSEE} yields that $\varsigma \leq \tau$ and $v= w$, $p_* = q_*$ on $\llbrak 0,\varsigma\llbrak$.\hfill $\qed$

\section{A Short Comment on Definition~\ref{df:fbp}}
\label{sec:motivation}
The main purpose of this discussion is to give, at least heuristically, a justification for the distribution valued terms which occur in the dynamics in Definition~\ref{df:fbp}. This is linked to the following question: Given the moving boundary problem in its description on each phases separately, does there exists a ``natural'' choice of distributional valued terms which vanish away from the interface? To be as simple as possible, we consider the following degenerate problem for $t> 0$,
\begin{gather}
  \label{eq:toy}
  \left\{\begin{split}
    \d v(t,x) &= 0, \qquad x\neq p_*(t), \\
    \ddx v(t,p_*(t)-) = \ddx v(t,p_*(t)+) &= 0,\\
    \d p_*(t) &= \sigma_* \d B_t,\\
    v_0(x) &= \1_{(0,\infty)}(x),\qquad x\neq 0,\\
    p_0 &= 0.
  \end{split}\right.
\end{gather}
The only meaningful solution of~\eqref{eq:toy} can be
\begin{equation}
  \label{eq:toysolution}
  v(t,x) = \1_{(p_*(t),\infty)} (x),\qquad t\geq 0,\; x\neq p_*(t).
\end{equation}
In fact, let $x \neq p_*(t)$, and consider the first hitting time
\[ \tau_x := \inf\{t\geq 0\,\vert\, p_*(t) = x\}.\]
Then, $t\mapsto v(t,x)$ has to be constant on $[0,\tau_x[$ and due to Neumann boundary conditions it is clear how to reiterate this procedure, starting at $\tau_x$ and considering $t\mapsto v(t,y)$ for $y\neq x$. 
On the other hand, let $\phi \in C^\infty_0(\R)$ and set $\Phi(x) := \int_{-\infty}^x \phi(y)\d y$. Then, by classical Ito-formula
\begin{align*}
  \scal{v(t,.)}{\phi}{}  -&\scal{v_0}{\phi}{}=\nonumber\\
                          &=  \Phi(0) - \Phi(p_*(t))  \label{eq:trafocomment} \\
                          &= - \int_0^t \phi(p_*(s)) \d p_*(s) -\frac12 \int_0^t \phi'(p_*(s)) \d\, [ p_* ](s) \nonumber \\   
                          &=-\int_0^t (v(s,p_*(s)+) - v(s,p_*(s)-)) \phi(p_*(s))\d p_*(s) \nonumber \\
                          &\qquad- \frac12 \int_0^t (v(s,p_*(s)+)- v(s,p_*(s)-)) \phi'(p_*(s)) \d\, [ p_* ](s) \nonumber \\
                          &=\int_0^t \scal{ L_1(v(s,.),p_*(s))}{ \phi}{}\d p_*(s) + \frac12 \int_0^t \scal{ L_2(v(s,.)}{p_*(s)}{} \d\, [ p_* ](s). \nonumber
\end{align*}
Recall that $\delta_p' \phi = - \phi'(p)$. Except from the
integrability conditions $v$ is indeed the solution of~\eqref{eq:toy}
in the sense of Definition~\ref{df:fbp}.

\begin{appendix}
\section{Abstract Setting}
\label{appendix}
We now briefly discuss existence and uniqueness results for stochastic evolution equations based on stochastic maximal $L^p$-regularity, which are due to van Neerven, Veraar and Weis~\cite{weisMaxRegEvEq}. 

On a separable Hilbert space $E$ we consider the stochastic evolution equation
\begin{equation}
  \label{eq:SEEabstract}
  \d X(t) = \left[A X(t) + B(X(t))\right] \d t + C(X(t))\d W_t,\; t\geq 0,
\end{equation}
with initial condition $X(0) = X_0$, where $W$ is a cylindrical Wiener process with covariance identity on another separable Hilbert space $U$. 
\begin{defn}\label{df:strongSEE}
  An $E$-predictable stochastic process $X$ is called local strong solution of~\eqref{eq:SEEabstract}, up to a predictable stopping time $\tau$, if $X(t)$ is $\dom(A)$-valued for a.\,a. $t>0$, and on $\llbrak 0,\tau\llbrak$
  \[ X(t) = X_0 + \int_0^t \left[A X(s) + B(X(s)) \right] \d s + \int_0^t C(X(s))\d W_s.\]
  In particular, all the integrals involved are assumed to exist on $E$, respectively as Bochner or stochastic integrals. 
\end{defn}
\begin{ass}
  \label{a:Aabst} $(-A,\dom(A))$ is densely defined and positive self-adjoint on $E$. 
\end{ass}
\begin{rmk}
  It is sufficient to assume that $-A$ has bounded $H^\infty$-calculus of angle $<\sfrac{\pi}2$. On Hilbert spaces, this is equivalent to the property that, after a possible change to an equivalent Hilbert space norm, $A$ generates an analytic $C_0$-semigroup of contractions~\cite[Section 7.3.3]{haase}.
\end{rmk}
For $\alpha\in \R$, we set 
\begin{equation}
  \label{eq:2}
  E_\alpha := \dom((-A)^\alpha), \qquad \norm{u}{\alpha} := \norm{(-A)^\alpha u}{E},\; u\in E_\alpha.
\end{equation}
Analogously, we denote the real interpolation spaces, for $1\leq p\leq \infty$, $\alpha\in (0,1)$, by
\[ E_{\alpha,p} := (E, \dom(A))_{\alpha,p},\]
and its respective norms by $\norm{.}{\alpha,p}$. Note that, since $A$ is negative self adjoint, for all $\alpha \in (0,1)$, cf. \cite[Thm 1.8.10]{triebel},
\begin{equation}
  \label{eq:complexreal}
  E_\alpha = E_{\alpha,2} = [E,\dom(A)]_\alpha,  
\end{equation}
with equivalence of norms, where the latter term denotes the complex interpolation space. Moreover, see~\cite[Thm 1.3.3]{triebel}, for $0<\alpha<\tilde \alpha<1$, $p$, $q\in [1,\infty)$, 
\begin{equation}
  \label{eq:realtheta12}
  E_1 \hookrightarrow E_{\tilde \alpha,p} \hookrightarrow E_{\alpha,q},
\end{equation}
and, if $1\leq p\leq q< \infty$, $\alpha\in (0,1)$, then
\begin{equation}
  \label{eq:realp12}
  E_{\alpha, p} \hookrightarrow E_{\alpha, q}.
\end{equation}
The density of the embeddings follows from \cite[Thm 1.6.2]{triebel}. If $q=\infty$, the embeddings are still continuous, but not dense, see~\cite[Rmk 1.18.3.]{triebel}.
\begin{example}
  Set $E:= L^2(\R^n)$, $A:= \Delta - \Id$ is the Laplacian with $\dom(A) := H^2(\R^n)$. Then, 
  \[ E_\alpha = H^{2\alpha}(\R^n)\quad\text{and}\quad E_{\alpha,p} = B^{2\alpha}_{2,p}(\R^n)\]
where $B^\alpha_{q,p}(\R)$ denote the Besov spaces for $\alpha\in (0,1)$, $1\leq p,q\leq \infty$, cf. Example 1.8 and 1.10 in~\cite{lunardiInterpol}.
\end{example}
\begin{ass}
  \label{a:Babst} $B =: B_1 + B_2$, where $B_1:  E_{1-\frac1p,p} \to E_{0}$ is Lipschitz continuous on bounded sets, and, there exists $L_{B}$ and $\tilde L_B$ such that for all $u$, $v\in E_1$,
  \[ \norm{B_2(u) - B_2(v)}{0} \leq L_B \norm{u-v}{1} + \tilde L_B \norm{u-v}{0}.\]
\end{ass}
\begin{ass}
  \label{a:Cabst} $C =: C_1 + C_2$, where $C_1:E_{1-\frac1p,p} \to \HS(U; E_{\sfrac12})$ is Lipschitz continuous on bounded sets and, there exists $L_{C}$ and $\tilde L_C$ such that for all $u$, $v\in E_{\sfrac12}$,
  \[ \norm{C_2(u) - C_2(v)}{\frac12} \leq L_C \norm{u-v}{1} + \tilde L_C \norm{u-v}{\frac12}.\]
\end{ass}
For the formulation of the existence theorem, denote by $M_p$ and $M_p^W$ the operator norms of
\[ g\mapsto \int_0^. S_{.-s} g(s) \d s,\qquad G\mapsto \int_0^. S_{.-s}G(s) \d W_s\]
as operators respectively from $L^p_{(\F_t)}(\R_+\times \Omega; E_0)$ into $L^p_{(\F_t)}(\R_+\times \Omega; E_1)$ and from $L^p_{(\F_t)}(\R_+\times \Omega; \HS(U; E_{\sfrac12}))$ into $L_{(\F_t)}^p(\R_+\times \Omega; E_1)$. Here, $L^p_{(\F_t)}(\R_+\times \Omega; E)$ denotes the $L^p$-space of all $\F_t$-adapted processes with values in $E$.
\begin{thm}
  \label{thm:eucSEEa}
  Let $p\geq 2$ and assume that Assumptions~\ref{a:Aabst}, \ref{a:Babst}, and~\ref{a:Cabst} hold true with
  \begin{equation}
    \label{eq:para}
    M_p L_B + M_p^W L_C < 1.    
  \end{equation}
  Then, for all $\F_0$ measurable initial data $X_0\in E_{1-\frac1p,p}$ there exists an unique maximal strong solution $(X,\tau)$ with values a.\,s. in
  \[L^p(0,\tau;E_1)\cap C([0,\tau); E_{1-\frac1p,p}).\]
  Moreover, almost everywhere on $\{\tau <\infty\}$,
  \[\lim_{t\nearrow \tau}\norm{X(t)}{1-\frac1p,p} = \infty.\]
\end{thm}
\begin{rmk}\label{rmk:para}
  It is known that that $M_2 \leq 1$, $M_2^W \leq \sfrac{1}{\sqrt 2}$, but in general $M_p$ and $M_p^W$ are not explicitly known for $p>2$. However, in some cases, like when $L_B = 0$ and $C_2$ is the generator of a unitary group then the theorem still holds true provided that $L_C <\sqrt{2}$, see \cite{veraarPara} for a detailed discussion of condition~\eqref{eq:para}. 
\end{rmk}

Because the results in \cite{weisMaxRegEvEq} are formulated on Banach spaces, where some additional difficulties occur, we shortly summarize the the arguments and results from the reference: Since $E$ is a separable Hilbert space, we get from Theorem 2.5, Remark 4.1.(v) and the discussions in section 5.2 and 5.3 in~\cite{weisMaxRegEvEq}, that in the situation of Theorem~\ref{thm:eucSEEa} all of the assumptions of the existence result \cite[Theorem 5.6]{weisMaxRegEvEq} are fulfilled. This gives the unique maximal mild solution $X$, which is also an analytically strong solution, see Proposition 4.4 and its proof in~\cite{weisMaxRegEvEq}.\\
By definition of maximal local solutions in \cite[Definition 5.5]{weisMaxRegEvEq}, we get
\[\lim_{t\nearrow \tau} \norm{X(t)}{E_{1-\frac1p,\frac1p}} = \infty,\qquad \text{on }\{\tau<\infty\}.\]

The corresponding result for global existence and additional regularity is extracted from part (ii) and (iii) of Theorem 5.6 in~\cite{weisMaxRegEvEq}.
\begin{thm}\label{thm:globalSEEabst}
  Assume that for $p\geq 2$ the conditions of Theorem~\ref{thm:eucSEEa} and linear growths assumptions on $B_1$ and $C_1$ are satisfied. Namely, there exists a $M>0$ such that 
  \begin{align*}
    \norm{B_1(u)}{0} &\leq M \left(1 + \norm{u}{1-\frac1p,p}\right),\\
    \norm{C_1(u)}{\HS(U;E_{\sfrac12})} &\leq M\left(1 + \norm{u}{1-\frac1p,p}\right), \qquad \forall u\in E_{1-\frac1p,p}.
  \end{align*}
  Let $(X,\tau)$ be the unique maximal solution of~\eqref{eq:SEEabstract}. Then, it holds that $\tau = \infty$ almost surely. If, moreover, $X_0 \in L^p(\Omega; E_{1-\frac1p,p})$, then, for all $T>0$ there exists a constant $K>0$ such that
  \begin{equation}
    \label{eq:18}
    \EE\int_0^T \norm{X(s)}{1}^p \d s + \E{\sup_{0\leq t\leq T} \norm{X(s)}{1-\frac1p,p}^p} \leq K \left(1 + \E{\norm{X_0}{1-\frac1p,p}^p}\right).
  \end{equation}
\end{thm}
\begin{rmk}\label{rmk:maximaility}
  The assumptions on the statements are indeed optimal in the following sense. Under the constraints of Assumption~\ref{a:Babst} the results cover even classes of fully-nonlinear equations. For the noise term, da Prato et al.~\cite{dPKZ} have shown that for an negative self-adjoint operator $A$, an element $b\in E$ and a real Brownian  motion $\beta$, there exists a strong solution of 
  \[\d X(t) = A X(t)\d t+ b \d \beta(t),\]
  if and only if $b\in E_{\sfrac12}$, see \cite[Thm. 6]{dPKZ}. For linear equations with multiplicative noise Brzezniak and Veraar~\cite{veraarPara} have discussed ill-posedness when~\eqref{eq:para} is violated.
\end{rmk}
\end{appendix}

\section*{Acknowledgements}
The author also would like to thank Martin Keller-Ressel and Wilhelm Stannat for comments and discussions and is grateful to Mark Veraar for pointing out the issue of Remark~\ref{rmk:para}.

\bibliographystyle{imsart-number}
\bibliography{AAP1359}

\begin{thebibliography}{40}

\bibitem{amannFunctionSpaces}
\begin{bbook}[author]
\bauthor{\bsnm{Amann},~\bfnm{H.}\binits{H.}}
(\byear{2009}).
\btitle{Anisotropic function spaces and maximal regularity for parabolic
  problems. {P}art 1}.
\bseries{Jind\u rich Ne\u cas Center for Mathematical Modeling Lecture Notes,
  6}.
\bpublisher{Matfyzpress, Prague}
\bnote{Function spaces}.
\bmrnumber{2907677}
\end{bbook}
\endbibitem

\bibitem{bloemkerCH}
\begin{barticle}[author]
\bauthor{\bsnm{{Antonopoulou}},~\bfnm{D.~C.}\binits{D.~C.}},
  \bauthor{\bsnm{{Bl{\"o}mker}},~\bfnm{D.}\binits{D.}} \AND
  \bauthor{\bsnm{{Karali}},~\bfnm{G.~D.}\binits{G.~D.}}
(\byear{2015}).
\btitle{{The Sharp interface limit for the stochastic Cahn-Hilliard equation}}.
\bjournal{ArXiv e-prints: 1507.07469}.
\end{barticle}
\endbibitem

\bibitem{katosqroot}
\begin{barticle}[author]
\bauthor{\bsnm{Auscher},~\bfnm{P.}\binits{P.}},
  \bauthor{\bsnm{McIntosh},~\bfnm{A.}\binits{A.}} \AND
  \bauthor{\bsnm{Nahmod},~\bfnm{A.}\binits{A.}}
(\byear{1997}).
\btitle{{The square root problem of {K}ato in one dimension, and first order
  elliptic systems}}.
\bjournal{Indiana Univ. Math. J.}
\bvolume{46}
\bpages{659--695}.
\bdoi{10.1512/iumj.1997.46.1423}
\bmrnumber{1488331}
\end{barticle}
\endbibitem

\bibitem{BarbuDaPratoStefan}
\begin{barticle}[author]
\bauthor{\bsnm{Barbu},~\bfnm{V.}\binits{V.}} \AND \bauthor{\bsnm{{Da
  Prato}},~\bfnm{G.}\binits{G.}}
(\byear{2002}).
\btitle{The two phase stochastic {S}tefan problem}.
\bjournal{Probability Theory and Related Fields}
\bvolume{124}
\bpages{544--560}.
\bdoi{10.1007/s00440-002-0232-4}
\end{barticle}
\endbibitem

\bibitem{bhqFunctional}
\begin{barticle}[author]
\bauthor{\bsnm{{Bayer}},~\bfnm{C.}\binits{C.}},
  \bauthor{\bsnm{{Horst}},~\bfnm{U.}\binits{U.}} \AND
  \bauthor{\bsnm{{Qiu}},~\bfnm{J.}\binits{J.}}
(\byear{2014}).
\btitle{{A Functional Limit Theorem for Limit Order Books with State Dependent
  Price Dynamics}}.
\bjournal{ArXiv e-print:1405.5230}.
\end{barticle}
\endbibitem

\bibitem{biaisLOBOrderFlow}
\begin{barticle}[author]
\bauthor{\bsnm{Biais},~\bfnm{B.}\binits{B.}},
  \bauthor{\bsnm{Hillion},~\bfnm{P.}\binits{P.}} \AND
  \bauthor{\bsnm{Spatt},~\bfnm{C.}\binits{C.}}
(\byear{1995}).
\btitle{{An Empirical Analysis of the Limit Order Book and the Order Flow in
  the Paris Bourse}}.
\bjournal{The Journal of Finance}
\bvolume{50}
\bpages{1655--1689}.
\end{barticle}
\endbibitem

\bibitem{bouchaudProfile}
\begin{barticle}[author]
\bauthor{\bsnm{Bouchaud},~\bfnm{J.~P.}\binits{J.~P.}},
  \bauthor{\bsnm{M{\'e}zard},~\bfnm{M.}\binits{M.}} \AND
  \bauthor{\bsnm{Potters},~\bfnm{M.}\binits{M.}}
(\byear{2002}).
\btitle{Statistical properties of stock order books: empirical results and
  models}.
\bjournal{Quantitative Finance}
\bvolume{2}
\bpages{251--256}.
\bdoi{10.1088/1469-7688/2/4/301}
\end{barticle}
\endbibitem

\bibitem{veraarPara}
\begin{barticle}[author]
\bauthor{\bsnm{Brze{\'z}niak},~\bfnm{Z.}\binits{Z.}} \AND
  \bauthor{\bsnm{Veraar},~\bfnm{M.}\binits{M.}}
(\byear{2012}).
\btitle{Is the stochastic parabolicity condition dependent on {$p$} and {$q$}?}
\bjournal{Electron. J. Probab.}
\bvolume{17}
\bpages{no. 56, 24}.
\bdoi{10.1214/EJP.v17-2186}
\bmrnumber{2955048}
\end{barticle}
\endbibitem

\bibitem{contHF}
\begin{barticle}[author]
\bauthor{\bsnm{Cont},~\bfnm{R.}\binits{R.}}
(\byear{2011}).
\btitle{Statistical Modeling of High-Frequency Financial Data}.
\bjournal{IEEE Signal Processing Magazine}
\bvolume{28}
\bpages{16--25}.
\bdoi{10.1109/MSP.2011.941548}
\end{barticle}
\endbibitem

\bibitem{contImbalance}
\begin{barticle}[author]
\bauthor{\bsnm{Cont},~\bfnm{R.}\binits{R.}},
  \bauthor{\bsnm{Kukanov},~\bfnm{A.}\binits{A.}} \AND
  \bauthor{\bsnm{Stoikov},~\bfnm{S.}\binits{S.}}
(\byear{2014}).
\btitle{The Price Impact of Order Book Events}.
\bjournal{Journal of Financial Econometrics}
\bvolume{12}
\bpages{47--88}.
\bdoi{10.1093/jjfinec/nbt003}
\end{barticle}
\endbibitem

\bibitem{cont2010obmodel}
\begin{barticle}[author]
\bauthor{\bsnm{Cont},~\bfnm{R.}\binits{R.}},
  \bauthor{\bsnm{Stoikov},~\bfnm{S.}\binits{S.}} \AND
  \bauthor{\bsnm{Talreja},~\bfnm{R.}\binits{R.}}
(\byear{2010}).
\btitle{A stochastic model for order book dynamics}.
\bjournal{Operations research}
\bvolume{58}
\bpages{549--563}.
\end{barticle}
\endbibitem

\bibitem{dPKZ}
\begin{barticle}[author]
\bauthor{\bsnm{{Da Prato}},~\bfnm{G.}\binits{G.}},
  \bauthor{\bsnm{Kwapie{\'n}},~\bfnm{S.}\binits{S.}} \AND
  \bauthor{\bsnm{Zabczyk},~\bfnm{J.}\binits{J.}}
(\byear{1987}).
\btitle{{Regularity of solutions of linear stochastic equations in {H}ilbert
  spaces}}.
\bjournal{Stochastics}
\bvolume{23}
\bpages{1--23}.
\bdoi{10.1080/17442508708833480}
\bmrnumber{920798}
\end{barticle}
\endbibitem

\bibitem{dPZinf}
\begin{bbook}[author]
\bauthor{\bsnm{{Da Prato}},~\bfnm{G.}\binits{G.}} \AND
  \bauthor{\bsnm{Zabczyk},~\bfnm{J.}\binits{J.}}
(\byear{2014}).
\btitle{Stochastic equations in infinite dimensions},
\bedition{second} ed.
\bseries{Encyclopedia of Mathematics and its Applications}
\bvolume{152}.
\bpublisher{Cambridge University Press, Cambridge}.
\bdoi{10.1017/CBO9781107295513}
\bmrnumber{3236753}
\end{bbook}
\endbibitem

\bibitem{bouchaudSMBP}
\begin{barticle}[author]
\bauthor{\bsnm{Donier},~\bfnm{J.}\binits{J.}},
  \bauthor{\bsnm{Bonart},~\bfnm{J.}\binits{J.}},
  \bauthor{\bsnm{Mastromatteo},~\bfnm{I.}\binits{I.}} \AND
  \bauthor{\bsnm{Bouchaud},~\bfnm{J.~P.}\binits{J.~P.}}
(\byear{2015}).
\btitle{A fully consistent, minimal model for non-linear market impact}.
\bjournal{Quantitative Finance}
\bvolume{15}
\bpages{1109--1121}.
\bdoi{10.1080/14697688.2015.1040056}
\end{barticle}
\endbibitem

\bibitem{lobsurvey}
\begin{barticle}[author]
\bauthor{\bsnm{{Gould}},~\bfnm{M.~D.}\binits{M.~D.}},
  \bauthor{\bsnm{{Porter}},~\bfnm{M.~A.}\binits{M.~A.}},
  \bauthor{\bsnm{{Williams}},~\bfnm{S.}\binits{S.}},
  \bauthor{\bsnm{{McDonald}},~\bfnm{M.}\binits{M.}},
  \bauthor{\bsnm{{Fenn}},~\bfnm{D.~J.}\binits{D.~J.}} \AND
  \bauthor{\bsnm{{Howison}},~\bfnm{S.~D.}\binits{S.~D.}}
(\byear{2013}).
\btitle{{Limit Order Books}}.
\bjournal{Quantitative Finance}
\bvolume{13}
\bpages{1709--1742}.
\end{barticle}
\endbibitem

\bibitem{grisvardCara}
\begin{barticle}[author]
\bauthor{\bsnm{Grisvard},~\bfnm{P.}\binits{P.}}
(\byear{1967}).
\btitle{Caract{\'e}risation de quelques espaces d'interpolation}.
\bjournal{Arch. Rational Mech. Anal.}
\bvolume{25}
\bpages{40--63}.
\bmrnumber{0213864 (35 \#\#4718)}
\end{barticle}
\endbibitem

\bibitem{haase}
\begin{bbook}[author]
\bauthor{\bsnm{Haase},~\bfnm{M.}\binits{M.}}
(\byear{2006}).
\btitle{The functional calculus for sectorial operators}.
\bseries{Operator Theory: Advances and Applications}
\bvolume{169}.
\bpublisher{Birkh{\"a}user Verlag, Basel}.
\bdoi{10.1007/3-7643-7698-8}
\bmrnumber{2244037}
\end{bbook}
\endbibitem

\bibitem{lehalle}
\begin{barticle}[author]
\bauthor{\bsnm{Huang},~\bfnm{W.}\binits{W.}},
  \bauthor{\bsnm{Lehalle},~\bfnm{C.~A.}\binits{C.~A.}} \AND
  \bauthor{\bsnm{Rosenbaum},~\bfnm{M.}\binits{M.}}
(\byear{2015}).
\btitle{Simulating and Analyzing Order Book Data: The Queue-Reactive Model}.
\bjournal{Journal of the American Statistical Association}
\bvolume{110}
\bpages{107--122}.
\bdoi{10.1080/01621459.2014.982278}
\end{barticle}
\endbibitem

\bibitem{js}
\begin{bbook}[author]
\bauthor{\bsnm{Jacod},~\bfnm{J.}\binits{J.}} \AND
  \bauthor{\bsnm{Shiryaev},~\bfnm{A.~N.}\binits{A.~N.}}
(\byear{2003}).
\btitle{Limit theorems for stochastic processes},
\bedition{second} ed.
\bseries{Grundlehren der Mathematischen Wissenschaften [Fundamental Principles
  of Mathematical Sciences]}
\bvolume{288}.
\bpublisher{Springer-Verlag, Berlin}.
\bdoi{10.1007/978-3-662-05265-5}
\bmrnumber{1943877 (2003j:60001)}
\end{bbook}
\endbibitem

\bibitem{SFBP}
\begin{barticle}[author]
\bauthor{\bsnm{Keller-Ressel},~\bfnm{M.}\binits{M.}} \AND
  \bauthor{\bsnm{M{\"u}ller},~\bfnm{M.~S.}\binits{M.~S.}}
(\byear{2016}).
\btitle{A {S}tefan-type stochastic moving boundary problem}.
\bjournal{Stochastics and Partial Differential Equations: Analysis and
  Computations}
\bpages{1--45}.
\bdoi{10.1007/s40072-016-0076-z}
\end{barticle}
\endbibitem

\bibitem{sowersEtAl2}
\begin{barticle}[author]
\bauthor{\bsnm{Kim},~\bfnm{K.}\binits{K.}},
  \bauthor{\bsnm{Mueller},~\bfnm{C.}\binits{C.}} \AND
  \bauthor{\bsnm{Sowers},~\bfnm{R.~B.}\binits{R.~B.}}
(\byear{2010}).
\btitle{{A stochastic moving boundary value problem}}.
\bjournal{Illinois J. Math.}
\bvolume{54}
\bpages{927--962}.
\bmrnumber{2928342}
\end{barticle}
\endbibitem

\bibitem{sowersEtAl}
\begin{barticle}[author]
\bauthor{\bsnm{Kim},~\bfnm{K.}\binits{K.}},
  \bauthor{\bsnm{Zheng},~\bfnm{Z.}\binits{Z.}} \AND
  \bauthor{\bsnm{Sowers},~\bfnm{R.~B.}\binits{R.~B.}}
(\byear{2012}).
\btitle{{A Stochastic Stefan Problem}}.
\bjournal{{Journal of Theoretical Probability}}
\bvolume{25}
\bpages{1040--1080}.
\bdoi{10.1007/s10959-011-0392-1}
\end{barticle}
\endbibitem

\bibitem{kirilenkoHFT}
\begin{barticle}[author]
\bauthor{\bsnm{Kirilenko},~\bfnm{Andrei~A.}\binits{A.~A.}},
  \bauthor{\bsnm{Kyle},~\bfnm{Albert~S.}\binits{A.~S.}},
  \bauthor{\bsnm{Samadi},~\bfnm{Mehrdad}\binits{M.}} \AND
  \bauthor{\bsnm{Tuzun},~\bfnm{Tugkan}\binits{T.}}
(\byear{2011}).
\btitle{{The Flash Crash: The Impact of High Frequency Trading on an Electronic
  Market}}.
\bjournal{Social Science Research Network Working Paper Series}.
\bdoi{10.2139/ssrn.1686004}
\end{barticle}
\endbibitem

\bibitem{krylov6}
\begin{bincollection}[author]
\bauthor{\bsnm{Krylov},~\bfnm{N.~V.}\binits{N.~V.}}
(\byear{1999}).
\btitle{An analytic approach to {SPDE}s}.
In \bbooktitle{Stochastic partial differential equations: six perspectives}.
\bseries{Math. Surveys Monogr.}
\bvolume{64}
\bpages{185--242}.
\bpublisher{Amer. Math. Soc., Providence, RI}.
\bdoi{10.1090/surv/064/05}
\bmrnumber{1661766 (99j:60093)}
\end{bincollection}
\endbibitem

\bibitem{krylovItoWentzell}
\begin{barticle}[author]
\bauthor{\bsnm{Krylov},~\bfnm{N.~V.}\binits{N.~V.}}
(\byear{2011}).
\btitle{On the {I}t\^o-{W}entzell formula for distribution-valued processes and
  related topics}.
\bjournal{Probab. Theory Related Fields}
\bvolume{150}
\bpages{295--319}.
\bdoi{10.1007/s00440-010-0275-x}
\bmrnumber{2800911 (2012i:60110)}
\end{barticle}
\endbibitem

\bibitem{lasrylions}
\begin{barticle}[author]
\bauthor{\bsnm{Lasry},~\bfnm{J.~M.}\binits{J.~M.}} \AND
  \bauthor{\bsnm{Lions},~\bfnm{P.~L.}\binits{P.~L.}}
(\byear{2007}).
\btitle{Mean field games}.
\bjournal{Jpn. J. Math.}
\bvolume{2}
\bpages{229--260}.
\bdoi{10.1007/s11537-007-0657-8}
\bmrnumber{2295621 (2008k:91034)}
\end{barticle}
\endbibitem

\bibitem{lionsmagenes1}
\begin{bbook}[author]
\bauthor{\bsnm{Lions},~\bfnm{J.~L.}\binits{J.~L.}} \AND
  \bauthor{\bsnm{Magenes},~\bfnm{E.}\binits{E.}}
(\byear{1972}).
\btitle{Non-homogeneous boundary value problems and applications - 1}.
\bpublisher{Springer}, \baddress{Springer}.
\end{bbook}
\endbibitem

\bibitem{liptonImbalance}
\begin{barticle}[author]
\bauthor{\bsnm{{Lipton}},~\bfnm{A.}\binits{A.}},
  \bauthor{\bsnm{{Pesavento}},~\bfnm{U.}\binits{U.}} \AND
  \bauthor{\bsnm{{Sotiropoulos}},~\bfnm{M.~G}\binits{M.~G.}}
(\byear{2014}).
\btitle{{Trading strategies via book imbalance}}.
\bjournal{Risk}.
\end{barticle}
\endbibitem

\bibitem{lunardiAnalytic}
\begin{bbook}[author]
\bauthor{\bsnm{Lunardi},~\bfnm{A.}\binits{A.}}
(\byear{1995}).
\btitle{{Analytic Semigroups and Optimal Regularity in Parabolic Problems}}.
\bseries{{Progress in Nonlinear Differential Equations and Their
  Applications}}.
\bpublisher{Birkh{\"a}user Basel}.
\end{bbook}
\endbibitem

\bibitem{lunardiMovingBoundary}
\begin{bincollection}[author]
\bauthor{\bsnm{Lunardi},~\bfnm{A.}\binits{A.}}
(\byear{2004}).
\btitle{{An Introduction to Parabolic Moving Boundary Problems}}.
In \bbooktitle{{Functional Analytic Methods for Evolution Equations}},
(\beditor{\bfnm{Mimmo}\binits{M.}~\bsnm{Iannelli}},
  \beditor{\bfnm{Rainer}\binits{R.}~\bsnm{Nagel}} \AND
  \beditor{\bfnm{Susanna}\binits{S.}~\bsnm{Piazzera}}, eds.).
\bseries{{Lecture Notes in Mathematics}}
\bvolume{1855}
\bpages{371--399}.
\bpublisher{Springer Berlin Heidelberg}.
\end{bincollection}
\endbibitem

\bibitem{lunardiInterpol}
\begin{bbook}[author]
\bauthor{\bsnm{Lunardi},~\bfnm{A.}\binits{A.}}
(\byear{2009}).
\btitle{Interpolation theory},
\bedition{second} ed.
\bseries{Appunti. Scuola Normale Superiore di Pisa (Nuova Serie). [Lecture
  Notes. Scuola Normale Superiore di Pisa (New Series)]}.
\bpublisher{Edizioni della Normale, Pisa}.
\bmrnumber{2523200 (2010d:46103)}
\end{bbook}
\endbibitem

\bibitem{markowichteichmannwolfram}
\begin{barticle}[author]
\bauthor{\bsnm{{Markowich}},~\bfnm{P.~A.}\binits{P.~A.}},
  \bauthor{\bsnm{{Teichmann}},~\bfnm{J.}\binits{J.}} \AND
  \bauthor{\bsnm{{Wolfram}},~\bfnm{M.~T.}\binits{M.~T.}}
(\byear{2016}).
\btitle{{Parabolic free boundary price formation models under market size
  fluctuations}}.
\bjournal{Multiscale Modeling \& Simulation}
\bvolume{14}
\bpages{1211--1237}.
\end{barticle}
\endbibitem

\bibitem{bouchaudReactionDiffusion}
\begin{barticle}[author]
\bauthor{\bsnm{Mastromatteo},~\bfnm{I.}\binits{I.}},
  \bauthor{\bsnm{T{\'o}th},~\bfnm{B.}\binits{B.}} \AND
  \bauthor{\bsnm{Bouchaud},~\bfnm{J.~P.}\binits{J.~P.}}
(\byear{2014}).
\btitle{Anomalous Impact in Reaction-Diffusion Financial Models}.
\bjournal{Phys. Rev. Lett.}
\bvolume{113}
\bpages{268701}.
\bdoi{10.1103/PhysRevLett.113.268701}
\end{barticle}
\endbibitem

\bibitem{diss}
\begin{bmisc}[author]
\bauthor{\bsnm{Mueller},~\bfnm{M.~S.}\binits{M.~S.}}
(\byear{2016}).
\btitle{Semilinear stochastic moving boundary problems}.
\bhowpublished{Doctoral thesis, TU Dresden}.
\end{bmisc}
\endbibitem

\bibitem{doubleauc}
\begin{barticle}[author]
\bauthor{\bsnm{Smith},~\bfnm{E.}\binits{E.}},
  \bauthor{\bsnm{Farmer},~\bfnm{J~D.}\binits{J.~D.}},
  \bauthor{\bsnm{Gillemot},~\bfnm{L.}\binits{L.}} \AND
  \bauthor{\bsnm{Krishnamurthy},~\bfnm{S.}\binits{S.}}
(\byear{2003}).
\btitle{Statistical theory of the continuous double auction}.
\bjournal{Quantitative finance}
\bvolume{3}
\bpages{481--514}.
\end{barticle}
\endbibitem

\bibitem{stefanEis}
\begin{barticle}[author]
\bauthor{\bsnm{Stefan},~\bfnm{J.}\binits{J.}}
(\byear{1888}).
\btitle{{{\"U}ber die Theorie der Eisbildung, insbesondere {\"u}ber die
  Eisbildung im Polarmeere.}}
\bjournal{{Wien. Ber. XCVIII, Abt. 2a (965--983)}}.
\end{barticle}
\endbibitem

\bibitem{triebel}
\begin{bbook}[author]
\bauthor{\bsnm{Triebel},~\bfnm{H.}\binits{H.}}
(\byear{1978}).
\btitle{Interpolation theory, function spaces, differential operators}.
\bseries{North-Holland Mathematical Library}
\bvolume{18}.
\bpublisher{North-Holland Publishing Co.}, \baddress{Amsterdam}.
\bmrnumber{503903 (80i:46032b)}
\end{bbook}
\endbibitem

\bibitem{weisMaxRegEvEq}
\begin{barticle}[author]
\bauthor{\bparticle{van} \bsnm{Neerven},~\bfnm{J.}\binits{J.}},
  \bauthor{\bsnm{Veraar},~\bfnm{M.}\binits{M.}} \AND
  \bauthor{\bsnm{Weis},~\bfnm{L.}\binits{L.}}
(\byear{2012}).
\btitle{Maximal {$L^p$}-regularity for stochastic evolution equations}.
\bjournal{SIAM J. Math. Anal.}
\bvolume{44}
\bpages{1372--1414}.
\bdoi{10.1137/110832525}
\bmrnumber{2982717}
\end{barticle}
\endbibitem

\bibitem{weisMaxReg}
\begin{barticle}[author]
\bauthor{\bparticle{van} \bsnm{Neerven},~\bfnm{J.}\binits{J.}},
  \bauthor{\bsnm{Veraar},~\bfnm{M.}\binits{M.}} \AND
  \bauthor{\bsnm{Weis},~\bfnm{L.}\binits{L.}}
(\byear{2012}).
\btitle{Stochastic maximal {$L^p$}-regularity}.
\bjournal{Ann. Probab.}
\bvolume{40}
\bpages{788--812}.
\bdoi{10.1214/10-AOP626}
\bmrnumber{2952092}
\end{barticle}
\endbibitem

\bibitem{zhengPhD}
\begin{bphdthesis}[author]
\bauthor{\bsnm{Zheng},~\bfnm{Z.}\binits{Z.}}
(\byear{2012}).
\btitle{{Stochastic Stefan Problems: Existence, Uniqueness and Modeling of
  Market Limit Orders}}
\btype{PhD thesis},
\bpublisher{Graduate College of the University of Illinois at
  Urbana-Champaign}.
\end{bphdthesis}
\endbibitem

\end{thebibliography}
\end{document}